\documentclass[10pt,oneside]{amsart}

\usepackage{amsmath}
\usepackage{amssymb}
\usepackage{amsthm}
\usepackage{textcomp}
\usepackage{rotating}
\usepackage{tabularx}
\usepackage{longtable}
\usepackage{array}
\usepackage{mathtools}
\usepackage{mathrsfs}
\usepackage{array}
\usepackage{tikz-cd}
\usepackage{adjustbox}
\usepackage{caption}
\usepackage[all]{xy}
\usepackage{booktabs}

\newcommand{\C}{\mathbb{C}}

\newcommand{\Q}{\mathbb{Q}}

\newcommand{\Qbar}{\overline{\mathbb{\Q}}}
\renewcommand{\P}{\mathbb{P}}

\DeclareMathOperator{\Pic}{Pic}
\DeclareMathOperator{\NS}{NS}

\newcommand{\magicsquarethree}[9]{%
\renewcommand{\arraystretch}{1.5}
  \begin{array}{|c|c|c|}\hline
    #1 & #2 & #3 \\\hline
    #4 & #5 & #6 \\\hline
    #7 & #8 & #9 \\\hline
  \end{array}%
}
\makeatletter
\newcommand{\magicsquarefour}[9]{%
  \def\@msqArow{#1 & #2 & #3 & #4}%
  \def\@msqBrow{#5 & #6 & #7 & #8}%
  \@magicsquarefourB{#9}%
}
\newcommand{\@magicsquarefourB}[8]{%
  \renewcommand{\arraystretch}{1.5}%
  \begin{array}{|c|c|c|c|}\hline
    \@msqArow \\\hline
    \@msqBrow \\\hline
    #1 & #2 & #3 & #4 \\\hline
    #5 & #6 & #7 & #8 \\\hline
  \end{array}%
}
\makeatother

\newcommand{\et}{\text{\'et}}

\newtheorem{theorem}{Theorem}[section]
\newtheorem{prop}[theorem]{Proposition}
\newtheorem{lemma}[theorem]{Lemma}

\newtheorem{theoremintro}{Theorem}

\theoremstyle{definition}
\newtheorem{defi}[theorem]{Definition}

\newtheorem{remark}[theorem]{Remark}

\newcommand{\linedef}[1]{\textbf{#1}}

\date{September 2026}

\usepackage{url}

\usepackage{hyperref}
\hypersetup{
    pdftitle={The Algebraic Geometry of 3 x 3 Magic Squares of Squares},
    pdfauthor={Asher Auel and Benjamin Singer},
    colorlinks=true,
    citecolor=blue,
    anchorcolor=blue,
    linkcolor=blue,
    urlcolor=blue
}

\begin{document}

\title[$3 \times 3$ magic squares of squares]{the algebraic geometry of $3 \times 3$ magic \\squares of squares}

\author{Asher Auel}
\address{Department of Mathematics, 
Dartmouth College, Kemeny Hall, Hanover, New Hampshire\\\texttt{\it E-mail
address: \tt asher.auel@dartmouth.edu}}

\author{Benjamin Singer}
\address{Department of Mathematics, 
Dartmouth College, Kemeny Hall, Hanover, New Hampshire\\\texttt{\it E-mail
address: \tt benjamin.d.singer.27@dartmouth.edu}}

\begin{abstract}
The question of whether a $3 \times 3$ magic square of squares with distinct integer entries exists has been open since the 18th century.  We study the geometry of the algebraic surface $V$ parameterizing $3 \times 3$ magic squares of squares, which is a singular complete intersection of six quadrics in $\mathbb{P}^8$. We compute its geometric automorphism group via an argument involving Gale duality. We compute the basic topological invariants and Hodge diamond of its resolution.  We provide an explicit rank 518 sublattice of its geometric Picard group, close to the Hodge-theoretic upper bound of 544. Finally, we study the geometry and arithmetic of del Pezzo, K3, and Enriques surfaces that arise as coordinate projections.
\end{abstract}

\maketitle


\section*{Introduction and History}

\subsection*{The History of Magic Squares}
An $n \times n$ magic square is a square with $n$ rows and $n$ columns with distinct positive integer entries such that each row, column, and diagonal all add up to the same number, often called the magic sum. Since a $2 \times 2$ magic square is impossible to construct, as the entries must all be equal, the smallest possible magic square is $3 \times 3$. The first instance of this was the Luo Shu magic square discovered in China nearly 3000 years ago. Chinese mythology tells of a turtle appearing out of the Luo river after a long famine, with markings on its shell in the following ordering \cite{cammann_1960_the}:
\[
\magicsquarethree{2}{7}{6}{9}{5}{1}{4}{3}{8}
\]
The rows, columns, and diagonals add to the magic sum 15, the smallest possible.

From a theoretical perspective, finding magic squares can be reduced to linear algebra. An $n \times n$ square with rational entries is determined by $2n+2$ linear equations in $n^2$ variables over $\mathbb{Q}$, giving infinitely many non-trivial integral solutions for all $n \geq 3$ by clearing denominators. For small $n$, the space can be explicitly parametrized. For example, in 1891, Lucas showed that every $3 \times 3$ magic square can be written as
\begin{equation}
    \label{eq:Lucas_param}
    \magicsquarethree{z-y}{x+y+z}{z-x}{y-x + z}{z}{x-y + z}{x + z}{z-x-y}{y+z}
\end{equation}
for some $x, y, z \in \mathbb{Z}$. In general, counting integral $n \times n$ magic squares is related to Ehrhart theory \cite{stanley1976magic,stanley1973linear}. See \cite{boyer_multimagiecom} as well as Section~\ref{subsec:automorphismgroup} for further details.

\subsection*{Magic Squares of Squares}

It has long been a fascination to construct magic squares under additional constraints. The most immediate of these is whether or not it is possible to construct a ``magic square of $k$th powers", an $n \times n$ magic square whose entries are entirely $k$th powers of integers for a given integer $k$. In the case $k = 2$, we will call an $n \times n$ magic square of $k$th powers a \linedef{magic square of squares} (or \linedef{square of squares}, for short). In 1770, Euler produced the $4 \times 4$ square of squares
\[
\magicsquarefour{68^2}{29^2}{41^2}{37^2}{17^2}{31^2}{79^2}{32^2}{59^2}{28^2}{23^2}{61^2}{11^2}{77^2}{8^2}{49^2}
\]
with magic sum 8515 in a letter to Lagrange \cite{boyer_2005_some}. Recently, Rome and Yamagishi~\cite{rome2025existence} showed that for each $k \geq 2$, there is an integer $n_0(k)$ such that $n \times n$ magic square of $k$th powers exists for any $n \geq n_0(k)$; in particular, an $n \times n$ magic square of squares exists for all $n \geq 4$. However, it is still unknown whether a $3 \times 3$ magic square of squares exists.

Many mathematicians have worked on $3 \times 3$ magic squares of squares in the last few centuries. Most notably, Lucas~\cite{lucas_1891_sur} proved that a $3 \times 3$ \textit{bimagic} square, i.e., a magic square that remains magic when all of its entries are squared, is impossible. A restatement of the $3 \times 3$ magic square of squares problem is due to LaBar~\cite{labar_1984_problem} in 1984, and has been maintained in the mathematical consciousness by many since: Gardner's column \cite{gardner1996magic}, Boyer's mathematical intelligencer article \cite{boyer_2005_some} and his upkeep of a website tracking progress on the problem \cite{boyer_multimagiecom}, and most recently by V\'arilly-Alvarado in a lecture at the Joint Mathematics Meetings in 2020 with a subsequent article \cite{vrillyalvarado_2021_the}.

While it is still unknown whether or not a $3 \times 3$ magic square of squares exists over the rational numbers, magic squares of squares of algebraic integers have been considered. The number fields of smallest degree where there exist known magic squares of squares are biquadratic extensions of $\mathbb{Q}$. For example, the square
\[
\magicsquarethree{(5-13\sqrt{3})^2}{(17 + 9\sqrt{3})^2}{(22 - 4\sqrt{3})^2}{(23 - \sqrt{3})^2}{(2\sqrt{133})^2}{(23 + \sqrt{3})^2}{(22 + 4\sqrt{3})^2}{(17 - 9\sqrt{3})^2}{(5 + 13\sqrt{3})^2}
\]
with magic sum 1596 over the biquadratic field $\mathbb{Q}(\sqrt{3}, \sqrt{133})$ is due to Bremner~\cite{bremner_1999_on} and is a magic square of squares over the totally real field of smallest degree known. For a positive integer $q$, there is also a general family 
\[
\magicsquarethree{(q^2 + 1)^2}{-(q^2 + 2q - 1)^2}{(2\sqrt{q^3 - q})^2}{-(q^2 - 2q - 1)^2}{0}{(q^2 - 2q - 1)^2}{-(2\sqrt{q^3 - q})^2}{(q^2 + 2q - 1)^2}{-(q^2 + 1)^2}
\]
over the field $\mathbb{Q}(i, \sqrt{q^3 - q})$ given in \cite{bremner_1999_on}. Of these, the magic square of squares over the number field of smallest discriminant known is
\[
\magicsquarethree{25}{-49}{24}{-1}{0}{1}{-24}{49}{-25}
\]
over $\mathbb{Q}(i, \sqrt{6})$, a member of the above family with $q = 2$. Similar families are given in Section~\ref{subsec:coordinate_hyperplane_sections}. Bremner~\cite{bremner_1999_on} also gave methods to construct magic squares of squares over number fields of odd degree. 

Progress on magic square of powers problems is tracked at \cite{boyer_multimagiecom}, including prizes up to \texteuro1500 and a bottle of champagne for solving problems related to $3 \times 3$ squares of squares. Gardner also placed a \$100 bounty on the problem \cite{gardner1996magic}. Further details can be found in \cite{boyer_2005_some}.

\subsection*{Results and Structure of the Paper}

This paper explores a perspective on the $3 \times 3$ magic square of squares problem through arithmetic geometry. The algebraic variety associated to $3 \times 3$ magic squares of squares is a complete intersection surface over $\Q$ with 256 $A_1$ singularities defined over $\Q(\zeta_8)$.  We will call this variety $V$, its resolution $\widetilde{V}$, and its complex analytification $\widetilde{V}^{an}$.  The $\Q$-rational points of $V$ yield integer $3 \times 3$ magic squares of squares by clearing denominators, but may have entries that are zero or not distinct.  The resolution $\widetilde{V}$ is of general type, so the sparsity of $3 \times 3$ magic squares of squares is consistent with the Bombieri--Lang conjecture on rational points of general type varieties (see Part F of~\cite{hindry2013diophantine} for an introduction). The $3 \times 3$ case is the only case where this happens; the variety associated to $4 \times 4$ magic squares of squares is a Calabi-Yau sevenfold and the varieties associated to $n \times n$ squares of squares for $n \geq 5$ are Fano, so they should contain many rational points. This perspective has appeared previously due to V\'arilly-Alvarado \cite{vrillyalvarado_2021_the}. Furthermore, Bruin, Thomas, and V\'arilly-Alvarado~\cite{bruin_2022_explicit} have shown $\widetilde{V}$ to be \textit{algebraically quasi-hyperbolic}, i.e., it contains finitely many curves of genus~0 and~1. Since these are the only curves that could have infinitely many rational points by Faltings' theorem, it confirms that rational points on $V$ are relatively sparse. Furthermore, in Appendix~\ref{subsec:lines-conics}, we prove that $V$ contains no lines over $\Qbar$. 

Arithmetic geometry also provides a potential solution to the $3 \times 3$ magic square of squares problem through the \'etale Brauer--Manin obstruction to rational points. Indeed, to find a magic square of squares with distinct entries is to find a rational point on an open subvariety $U \subset \widetilde{V}$. Computing the Brauer group of $\widetilde{V}$ and then the \'etale Brauer--Manin set of $U$ could provide an obstruction, resolving the problem (see \cite{viray2023rational,jeanlouiscolliotthlne_2021_the} for an overview of this approach). However, accomplishing this requires understanding the geometry, topology, and Hodge theory of $\widetilde{V}^{an}$, the geometric Picard group and its structure as a Galois module, and the geometry of the complement of $U$. Our work provides the first steps toward realizing this approach. 

The variety of magic squares has a large symmetry group, which we compute in the following.

\begin{theoremintro}[Geometric Automorphism Group]
\label{thm:main3}
The geometric automorphism group $\mathrm{Aut}(V_{\overline{\mathbb{Q}}}) \cong \mathrm{Aut}(\widetilde{V}_{\overline{\mathbb{Q}}})$ is isomorphic to the finite group $G \cong D_4 \ltimes (\mu_2^9)/\mu_2$ of order $2048$. 
Its action on $V$ is by coordinate sign changes in the ambient $\mathbb{P}^8$ and symmetries of the square, and its action on $\widetilde{V}$ is obtained by lifting to the minimal resolution.
\end{theoremintro}

We prove Theorem~\ref{thm:main3} in Section \ref{subsec:automorphismgroup}. 
The automorphism group clearly contains the group $G$, which fits into a split exact sequence
\[
1 \longrightarrow \mu_2^9/\mu_2 \longrightarrow G \longrightarrow D_4 \longrightarrow 1.
\]
The subgroup of sign-changes acts trivially on the entries of the associated magic square (obtained by squaring the coordinates of points of $V$), and the dihedral group $D_4$ of order 8 acts non-trivially on the magic square.

The content of the theorem lies in showing the reverse inclusion $\mathrm{Aut}(V_{\overline{\mathbb{Q}}}) \leq G$. We show that every automorphism of $V$ is induced by an element of $\mathrm{PGL}_9$ preserving the six-dimensional space $W$ of defining quadrics, and hence projectively permutes the nine coefficient functionals in $\mathbb{P}(W^{\vee})$. The space of linear relations among these functionals recovers the three-dimensional space of $3 \times 3$ magic squares parametrized by Lucas, whose Gale transform is a $3 \times 3$ grid in $\mathbb{P}^2$. The rigidity of the line arrangements within the grid forces the image of $\mathrm{Aut}(V_{\overline{\mathbb{Q}}})$ in $\mathrm{PGL}(W)$ to be $D_4$. A separate argument identifies the kernel with the group of sign changes.

We also compute the main topological and Hodge-theoretic information about the (resolution of the) variety of squares of squares.

\begin{theoremintro}[Hodge Numbers, Betti Numbers, and Chern Numbers]\label{thm:main1}
The complex manifold $\widetilde{V}^{an}$ is simply connected and has Hodge diamond
\[
\begin{array}{ccccc}
 &  & 1 &  & \\
 & 0 &  & 0 & \\
111 &  & 544 &  & 111\\
 & 0 &  & 0 & \\
 &  & 1 &  &
\end{array}
\]
This gives Betti numbers $b_0 = b_4 = 1$, $b_1 = b_3 = 0$, $b_2 = 766$. Furthermore, the singular cohomology of $\widetilde{V}^{\mathrm{an}}$ is torsion-free, and the Chern numbers are $c_1^2(\widetilde{V}^{\mathrm{an}}) = 576$ and $c_2(\widetilde{V}^{\mathrm{an}}) = 768$. The lattice $H^2(\widetilde{V}^{\mathrm{an}}, \mathbb{Z})$ with the intersection pairing is odd and unimodular of signature $(223, 543)$, hence isomorphic to $\langle 1 \rangle^{\oplus 223} \oplus \langle -1 \rangle^{\oplus 543}$.
\end{theoremintro}

We prove Theorem~\ref{thm:main1} by showing that $\widetilde{V}^{an}$ is diffeomorphic to a smooth complete intersection of type $(2, 2, 2, 2, 2, 2)$ in $\mathbb{P}^8$ via a result of Atiyah \cite{atiyah_1958_on}; the invariants follow quickly from there.

We also compute bounds on the Picard group of the (resolution of the) variety of magic squares of squares.

\begin{theoremintro}[Picard Rank]
\label{thm:picard-rank}
   The resolution $\widetilde{V}$ has torsion-free geometric Picard group, with geometric Picard rank $518 \leq \rho(\widetilde{V}_{\overline{\mathbb{Q}}}) \leq 544$.
\end{theoremintro}

In Sections \ref{sec:divisors} and \ref{sec:magic-surfaces}, we produce 1204 divisors on $\widetilde{V}_{\Qbar}$ including the 256 exceptional divisors, 416 split hyperplane sections over several number fields, and $384$ and $148$ preimages of lines on quartic del Pezzo surfaces and cubic surfaces (respectively) that admit a dominant morphism from $V$. In particular, we explicitly describe the complement of $U$ in $\widetilde{V}$. Since $\widetilde{V}$ is simply connected by Theorem \ref{thm:main1}, it admits no nontrivial connected finite \'etale covers. Thus any \'etale input for a potential \'etale Brauer--Manin obstruction on $U$ must come from covers ramified along its complement.

In Section~\ref{sec:PicardLattice}, we give an account of intersection theory on $V$, using the data accumulated from the previous sections to produce a Galois-stable rank 518 sublattice of the Picard group. The upper bound is yielded by Theorem~\ref{thm:main1} and the Lefschetz theorem on $(1, 1)$-classes, and torsion-freeness is yielded by a standard argument with the exponential exact sequence. In the process of searching for divisors, we find a new connection between $V$ and the congruent number problem, as the elliptic curve $y^2 = x^3 - x$ appears in the decomposition of one of these hyperplane sections. A different connection was previously found by Robertson~\cite{robertson_1996_magic}. Furthermore, we give new parametrizations of $3 \times 3$ magic squares of squares over biquadratic fields, producing a square over $\mathbb{Q}(i, \sqrt{5})$, a square over a field of smaller discriminant than Bremner's example over $\mathbb{Q}(i, \sqrt6)$. 

Finally, a subproblem of finding a $3 \times 3$ magic square of squares is to find ${3 \times 3}$ magic squares with a certain number of square entries. In \cite{bremnersquaresii}, Bremner recasts this problem geometrically, showing that squares of 6 squares are governed by degree 8 K3 surfaces, and he extensively studies one of these K3 surfaces. We extend Bremner's perspective in three directions in Section \ref{sec:magic-surfaces}. Firstly, we explore projections of $V$ to quartic del Pezzo surfaces, which parametrize squares of 5 squares. Secondly, we show that one of the K3 surfaces has geometric Picard rank 19 and write down explicit equations for the others. Finally, we use the explicit geometry of the K3 surface that Bremner studies to produce other related surfaces: a family of twisted derived equivalent K3 surfaces of degree 12, and a family of Enriques surfaces that it covers.

\subsection*{Acknowledgments}

The authors would like to thank Nick Addington, Phil Engel, Sarah Frei, Salim Tayou, Tony V\'arilly-Alvarado, and John Voight for helpful comments. The second author would especially like to thank John Voight, whose talk on the problem at a meeting of the Dartmouth Mathematical Society in 2024 inspired this project. This paper is also heavily inspired by the work of van Luijk and of Stoll and Testa on the perfect cuboid problem \cite{van2000perfect,stoll2010surface}. In fact, our Magma \cite{bosma1997magma} scripts to compute ranks of Picard sublattices later in the paper are modified versions of Stoll's code to compute intersection matrices of blown-up surfaces.

The first author was supported by NSF grant DMS-2200845.  The second author was supported by a leave term grant from the office of Scholars Programs, Undergraduate Research, and Fellowships at Dartmouth College, part-time research grants from the James O.\ Freedman Presidential Scholars program, and through the Jack Byrne Scholars Program in Math and Society.  The second author presented preliminary versions of some of these results at a poster session at the AGNES conference at Dartmouth College in November 2024, and during a talk at the Young Mathematicians Conference at the Ohio State University in July 2025.

\subsection*{Statement on AI Usage} The second author used Claude Opus 5, courtesy of Dartmouth College, in the editing process to check for typos, notational inconsistencies, and errors, as well as to clean up code to make it more presentable. All mathematics and writing within this paper are entirely the work of the authors.

\section{The Variety Parameterizing {$3 \times 3$} Magic Squares of Squares}

\subsection{Equations and Basic Details} We will write a $3 \times 3$ magic square of squares as the following:
\[
\magicsquarethree{A^2}{B^2}{C^2}{D^2}{M^2}{E^2}{F^2}{G^2}{H^2}
\]
Denote the special central entry as $M^2$. Writing the magic sum as $N$ yields
\begin{align*}
3N &{} = F^2 + M^2 + C^2 + H^2 + M^2 + A^2 + G^2 + M^2 + B^2  \\
   &{} = 3M^2 + (A^2 + B^2 + C^2) + (F^2 + G^2 + H^2)  \\
   &{} = 3M^2 + 2N,
\end{align*}
which implies that $N = 3M^2$. Thus, we can define $V$ $\subset \mathbb{P}^8$ as the variety cut out by the equations
\begin{align*}
    A^2 + H^2 - 2M^2 = &{\,} 0, &{} \quad A^2 + B^2 + C^2 - 3M^2 = &{\,} 0, \\
    B^2 + G^2 - 2M^2 = &{\,} 0, &{} \quad A^2 + D^2 + F^2 - 3M^2 = &{\,} 0, \\
     C^2 + F^2 - 2M^2 = &{\,} 0, &{} \quad C^2 + E^2 + H^2 - 3M^2 = &{\,} 0, \\
    D^2 + E^2 - 2M^2 = &{\,} 0, &{} \quad F^2 + G^2 + H^2 - 3M^2 = &{\,} 0.   
\end{align*}
However, for our purposes, it is convenient to work with the Gr{\"o}bner basis of the defining ideal 
\begin{align*}
    3A^2 - 2F^2 - 2G^2 + H^2 = &{\, } 0, &{} \quad 3D^2 + 2F^2 -G^2 - 4H^2 = &{\,} 0,\\
    3B^2 - 2F^2 + G^2 - 2H^2 = &{\, }0, &{} \quad  3E^2 - 4F^2 - G^2 + 2H^2 = &{\,} 0,\\
    3C^2 + F^2 - 2G^2 - 2H^2 = &{\,} 0, &{} \quad 3M^2 - F^2 - G^2 - H^2 = &{\,} 0.
\end{align*}
This gives us a morphism
\begin{align*}
\pi : V &{} \longrightarrow \mathbb{P}^2\\
[A:B:C:D:M:E:F:G:H] &{} \longmapsto [F:G:H]
\end{align*}
of degree 64 with fibers producing the square
\[
\magicsquarethree{\frac{2}{3}F^2 + \frac{2}{3}G^2 - \frac{1}{3}H^2}{\frac{2}{3}F^2 - \frac{1}{3}G^2 + \frac{2}{3}H^2}{-\frac{1}{3}F^2 + \frac{2}{3}G^2 + \frac{2}{3}H^2}{-\frac{2}{3}F^2 + \frac{1}{3}G^2 + \frac{4}{3}H^2}{\frac{1}{3}F^2 + \frac{1}{3}G^2 + \frac{1}{3}H^2}{\frac{4}{3}F^2 + \frac{1}{3}G^2 - \frac{2}{3}H^2}{F^2}{G^2}{H^2}
\]
The exact same process but with any other outer row or column will yield isomorphic descriptions of $V$ as an intersection of $6$ quadrics in $\mathbb{P}^8$.

From the Gr{\"o}bner basis, $V$ is two-dimensional, and its codimension is equal to the number of equations which determine it. This implies that it is a complete intersection of type $(2, 2, 2, 2, 2, 2)$ in $\mathbb{P}^8$, and a surface of degree 64.  $V$ has 256 ordinary double point singularities all defined over the number field $\mathbb{Q}(\sqrt{2}, i)$ \cite{bruin_2022_explicit} explicitly described in \ref{subsec:trivialpoints}, so it is arithmetically Cohen--Macaulay and thus (projectively) normal. Furthermore, since positive-dimensional complete intersections are connected, $V$ is irreducible. This implies the following:

\begin{lemma}
\label{lemma:gentype}
The resolution $\widetilde{V}$ is of general type.
\end{lemma}

\begin{proof}
    Since $V$ is a Cohen--Macaulay complete intersection of type $(2, 2, 2, 2, 2, 2)$ with sum of degrees $12$, we have that the dualizing sheaf $\omega^{\circ}_{V}$ is isomorphic to $ \mathscr{O}_{V}(12- 9) = \mathscr{O}_{V}(3)$ by adjunction. Since $V$ has $A_1$ singularities which are resolved crepantly, denoting the blow-up map $\pi\colon \widetilde{V}\to V$, we have that the canonical bundle $\omega_{\widetilde{V}} $ is isomorphic to $\pi^*\omega^{\circ}_{V}$ (see \cite{reid1987young}). Since $\omega^{\circ}_{V}$ is big, $\pi^*\omega^{\circ}_{V}$ must be big by Proposition 2.2.43 in \cite{lazarsfeld2017positivity}. Hence, $\widetilde{V}$ is of general type.
\end{proof} 

The Bombieri--Lang conjecture \cite{hindry2013diophantine} predicts that the rational points of $\widetilde{V}$ and $V$ are contained in a Zariski closed subset, i.e., a finite union of curves and points. This would guarantee some open subset $V' \subset V$ with no rational points. However, this is still not enough to solve the $3 \times 3$ magic square of squares problem; since a square having distinct and nonzero entries is an open condition, solving the problem is equivalent to showing that a specific open subset $U \subset V$ has no rational points. 

\begin{lemma}
Let $U \subset V$ be the locus corresponding to magic squares with distinct (but possibly zero) entries.  Then $U$ is a Zariski open subscheme of $V$ and the $\Q$-rational points of $U$ correspond to $3 \times 3$ magic squares of squares with distinct positive integer entries.
\end{lemma}
\begin{proof}
Since $U \subset V$ is the complement of a finite number of hyperplane sections, e.g.\ $A-B=0$, etc., it is Zariski open.  For the final claim, by clearing denominators, it suffices to remark that $V$ admits no $\Q$-rational points with distinct coordinates and one of the coordinates equal to zero.  This follows from a case-by-case analysis detailed in Section~\ref{subsec:coordinate_hyperplane_sections}.
\end{proof}

We completely analyze the complement of this open subvariety in Section~\ref{subsec:nondistinct}.

\begin{remark}
A different approach to defining equations for magic squares of squares is given in \cite{bremner_1999_on, bremnersquaresii}, though it still produces the same variety $V$. In Bremner's work, he uses the Lucas parametrization \eqref{eq:Lucas_param} 
setting each expression to a square (e.g. $z - y = A^2$). Geometrically, this corresponds to a complete intersection of 9 multi-homogeneous quadrics in weighted projective space 
\[
X \subset \mathbb{P}(2, 2, 2, 1, 1, 1, 1, 1, 1, 1, 1, 1)
\]
The variety $X$ has two canonical linear projections. The first is a projection $X \to \mathbb{P}^8$ defined by
\[
[x:y:z:A:B:C:D:M:E:F:G:H] \longmapsto [A:B:C:D:M:E:F:G:H],
\]
yielding the associated $3 \times 3$ magic square of squares; the image of this projection is $V$. The second is a projection $X \to \mathbb{P}^2$ defined by
\[
[x:y:z:A:B:C:D:M:E:F:G:H] \longmapsto [x:y:z],
\]
yielding the Lucas parametrization of the magic square; this is a finite cover of degree 256. This perspective can also be used to pass between Bremner's defining equations for K3 surfaces parametrizing squares of 6 squares in \cite{bremnersquaresii} and our descriptions in Section~\ref{subsec:magic_K3}.
\end{remark} 

\subsection{Automorphism Group}
\label{subsec:automorphismgroup}

The automorphism groups of varieties of general type are known to be finite \cite{hacon2013birational}. In addition, the general complete intersection of type $(2, 2, 2, 2, 2, 2)$ in projective space is known to have trivial automorphism group. A full proof of this folklore result can be found in \cite{chen2024automorphism}, with preliminary work done by Benoist~\cite{benoistseparation}. However, due to the symmetry of its construction, the variety $V$ has a large finite automorphism group. There are two subgroups of interest. Firstly, since a magic square of squares 
\[
\magicsquarethree{A^2}{B^2}{C^2}{D^2}{M^2}{E^2}{F^2}{G^2}{H^2}
\]
has square entries, changing the sign of a coordinate of the point $[A:\cdots :H] \in V$ will still yield a magic square of squares. This yields $9$ involutions $\sigma_A, \dots, \sigma_H$ of $V$. This corresponds to a subgroup $S \cong (\mu_2)^9/\mu_2 \leq \mathrm{Aut}(V)$, where we quotient out by the diagonal negation action which is trivial in projective space. 

We also notice that there is a subgroup $T \cong D_4 \leq \mathrm{Aut}(V)$ corresponding to the symmetries of the square. The rotation is given by
\begin{align*}
 \rho : \mathbb{P}^8 &{} \longrightarrow \mathbb{P}^8\\ 
 [A:B:C:D:M:E:F:G:H] &{} \longmapsto [F:D:A:G:M:B:H:E:C]
\end{align*}
with corresponding operation 
\[
\magicsquarethree{A^2}{B^2}{C^2}{D^2}{M^2}{E^2}{F^2}{G^2}{H^2} \longmapsto \magicsquarethree{F^2}{D^2}{A^2}{G^2}{M^2}{B^2}{H^2}{E^2}{C^2} 
\]
on the magic square. The reflection is given by
\begin{align*}
 \psi : \mathbb{P}^8 &{} \longrightarrow \mathbb{P}^8\\ [A:B:C:D:M:E:F:G:H] &{} \longmapsto [C:B:A:E:M:D:H:G:F]
\end{align*}
with corresponding operation
\[
\magicsquarethree{A^2}{B^2}{C^2}{D^2}{M^2}{E^2}{F^2}{G^2}{H^2} \longmapsto \magicsquarethree{C^2}{B^2}{A^2}{E^2}{M^2}{D^2}{H^2}{G^2}{F^2}
\]
on the magic square. The group $T$ acts on $S$ by conjugation, yielding the semidirect product subgroup $G = T \ltimes S \leq \mathrm{Aut}(V)$ of order $256\cdot 8  = 2048$.

We now prove Theorem~\ref{thm:main3}, showing that $G$ is the geometric automorphism group. Before proceeding, we will need a statement about the behavior of finite geometric automorphism groups under base change, which is well known but lacks a citable reference. We prove it for convenience: 

\begin{lemma}
\label{lemma:automorphismgrouptechnicallemma}
    Let $X$ be a smooth projective variety over $\mathbb{Q}$ with finite geometric automorphism group. Then, $\mathrm{Aut}_{\Qbar}(X_{\Qbar}) \cong \mathrm{Aut}_{\C}(X_{\C})$, where $X_{\Qbar}$ and $X_{\C}$ are the base changes of $X$ to $\Qbar$ and $\C$ respectively.
\end{lemma}

\begin{proof}
    By \cite{matsumura1967representability}, the geometric automorphism group of $X$ is represented by a finite type group scheme over $\Q$, which we denote by $\mathcal{G}$. However, since the geometric automorphism group of $X$ is finite, $\mathcal{G}(\Qbar)$ is finite. Hence $\mathcal{G}$ must be 0-dimensional and thus finite. Furthermore, since $X$ is defined over a field of characteristic zero, by Cartier's Theorem, $\mathcal{G}$ must also be smooth, hence reduced. Thus $\mathbb{G}$ is a finite \'etale group scheme, and so $\mathrm{Aut}_{\Qbar}(X_{\Qbar}) \cong \mathcal{G}(\Qbar) \cong \mathcal{G}(\C) \cong \mathrm{Aut}_{\C}(X_\C)$ as desired.  
\end{proof}

\begin{proof}[Proof of Theorem~\ref{thm:main3}]
Since $V$ is a normal surface, every automorphism lifts uniquely along the minimal resolution $\pi\colon \widetilde V \to V$, so $\operatorname{Aut}(V) \cong \operatorname{Aut}(\widetilde V)$. By Lemma \ref{lemma:automorphismgrouptechnicallemma}, it suffices to prove the result over $\mathbb{C}$. For convenience, throughout the proof, we write $V$ for the base change $V_\C$ (similarly for $\widetilde{V}$ as $\widetilde{V}_\C$).

By the proof of Lemma~\ref{lemma:gentype}, we have $\omega_{\widetilde V} \cong \pi^{*}\mathscr{O}_V(3)$. Every automorphism of $\widetilde V$ preserves $\omega_{\widetilde V}$, and since $\operatorname{Pic}(\widetilde V)$ is torsion free by Proposition~\ref{prop:pic-ns}, it preserves the unique cube root $\pi^{*}\mathscr{O}_V(1)$. Since $V$ is projectively normal and the Gr\"obner basis shows that its defining ideal contains no linear forms, restriction gives $H^0(\mathbb{P}^8, \mathscr{O}_{\mathbb{P}^8}(1)) \cong H^0(V, \mathscr{O}_V(1))$, so $h^0(V, \mathscr{O}_{V}(1)) = h^0(\mathbb{P}^8, \mathscr{O}_{\mathbb{P}^8}(1)) = 9$. Hence, by the fact that $\pi_*\mathscr{O}_{\widetilde{V}} \cong \mathscr{O}_V$ and the projection formula, we have $h^{0}(\widetilde V, \pi^{*}\mathscr{O}_V(1)) = h^{0}(V, \mathscr{O}_V(1)) = 9$. Since the associated morphism factors as $\widetilde V \xrightarrow{\pi} V \hookrightarrow \mathbb{P}^{8}$, every automorphism is induced by an element of $\operatorname{PGL}_{9}(\C)$ preserving $V$. Since $V$ is a complete intersection, whose ideal is generated in degree two, these are exactly the elements preserving the six-dimensional space $W \subset H^{0}(\mathbb{P}^{8}, \mathscr{O}(2))$ spanned by the defining quadrics. Thus $\operatorname{Aut}(V) = \operatorname{Stab}_{\operatorname{PGL}_{9}(\C)}(W)$. Since $G \subseteq \operatorname{Stab}_{\operatorname{PGL}_{9}(\C)}(W)$, it remains to prove the reverse inclusion.

The defining quadrics are diagonal. So, for each coordinate $v \in \{A, \dots, H\}$, the coefficient of $v^{2}$ defines a linear functional $c_{v} \colon W \to \mathbb{C}$. Restriction to $W$ gives a homomorphism $\operatorname{Stab}_{\operatorname{PGL}_{9}}(W) \to \operatorname{PGL}(W)$, whose kernel we call $K$ and whose image we call $\Gamma$. We will show $K = S$ and $\Gamma = T$.

Let $W^{\perp} \subseteq \mathbb{C}^{9}$ denote the space of linear relations $\sum_{v} \lambda_{v} c_{v} = 0$ among the nine functionals. A vector $(\lambda_{v})$ is in $W^{\perp}$ when it is orthogonal to every defining quadric. In other words, when $(\lambda_{v})$ is viewed as a $3 \times 3$ array, it has all three rows, three columns and two main diagonals summing to three times its center. That is, $W^{\perp}$ is the three-dimensional space of $3 \times 3$ magic squares.

To show $K = S$, we first note that the nine functionals $c_{v}$ are pairwise linearly independent. Now, if $g \in K$, then after scaling, $g$ fixes $W$ pointwise, so $g^{\top} q\, g = q$ for every $q \in W$. A general $q \in W$ is nondegenerate, as each $c_{v}$ is nonzero. Hence, fixing such a $q$, the element $g$ commutes with every $q^{-1}q'$ ($q' \in W$), a diagonal matrix with entries $c_{v}(q')/c_{v}(q)$. Since the $c_v$'s are pairwise linearly independent, these separate the nine coordinates. Hence, $g$ is diagonal, and a diagonal element fixing $W$ has all entries $\pm 1$. This gives $K = (\mu_{2})^{9}/\mu_{2} = S$.

Now, we show that $\Gamma = T$. The key point is that the configuration of the nine coefficient functionals remembers the original magic square combinatorics. We will make this precise using Gale duality.

A linear automorphism preserves the rank of a quadric, hence the discriminant locus of $W$. Since $W$ is a system of diagonal quadrics, the discriminant locus is the union of the nine hyperplanes $\{c_{v} = 0\} \subset \mathbb{P}(W)$. Since the $c_{v}$ are distinct, this amounts to projectively permuting the nine points $[c_{v}] \in \mathbb{P}(W^{\vee}) = \mathbb{P}^{5}$. The space of linear relations among the $c_{v}$ is $W^{\perp}$, the magic squares. Choose the basis $J, N_{1}, N_{2}$ of $W^{\perp}$, where $J$ is the constant $3 \times 3$ matrix
\[
J = \begin{pmatrix}
    1 & 1 &1 \\
    1 & 1 & 1\\
    1 & 1 & 1
\end{pmatrix}
\]
and
\[
  N_{1} = \begin{pmatrix} 1 & -1 & 0 \\ -1 & 0 & 1 \\ 0 & 1 & -1 \end{pmatrix},
  \qquad
  N_{2} = \begin{pmatrix} 0 & 1 & -1 \\ -1 & 0 & 1 \\ 1 & -1 & 0 \end{pmatrix}
\]
are sum-zero. Defining
\[
\mathrm{ev}_v\colon W^{\perp} \to \mathbb{C}
\]
to be projection to the $v$th coordinate on $W^{\perp}$, the Gale transform \cite{gale1956neighboring} of the configuration $\{c_v\}_v$ is the configuration of the projective points 
\[
\{[\mathrm{ev}_v(J): \mathrm{ev}_v(N_1): \mathrm{ev}_v(N_2)]\}_v = \{[1:y_v:z_v]\}_v \in \mathbb{P}^2.
\]
Here, $(y_v, z_v)$ is a point of the $3 \times 3$ grid $\{-1, 0, 1\}^2$, and the center $M$ is at the origin. The projective symmetries of a configuration of points and their Gale transforms coincide, see \cite{eisenbud2000projective, dolgachev1988point, dolgachev2012classical}. Hence, $\Gamma$ is identified with the projective automorphisms of the grid $\{-1, 0, 1\}^2$. A projective transformation permuting the nine points must preserve all collinear triples. For the grid $\{-1, 0, 1\}^2$ these are the three rows, the three columns, and the two diagonals. We note that this labelling does not line up with the original magic square on the nose: it corresponds to
\[
\magicsquarethree{B}{F}{E}{H}{M}{A}{D}{C}{G}
\]
Due to this, the collinear triples correspond to the four magic sums through the center and the broken diagonals.

Any projective transformation must fix the center $M$, which is on four lines, preserve the set $\{B, D, E, G\}$ of points on three lines, and the set $\{A, C, F, H\}$ of points on two, while permuting the former as a symmetry of the square they span. This implies that it is determined by rotations or reflections of that square, i.e. it is an element of $D_4$. Hence, $\Gamma \subseteq D_4$. Since $T \cong D_4$ maps isomorphically onto its image under the projection $\mathrm{Aut}(V) \to \Gamma$, we have that $\Gamma = T$.
 
Finally, we have that $\lvert \mathrm{Aut}(V) \rvert = \lvert K \rvert \cdot \lvert\Gamma \rvert = 256 \cdot 8 = 2048 = \lvert G \rvert$, so $\mathrm{Aut}(V) = \mathrm{Aut}(\widetilde{V}) = G$. 
\end{proof}

\begin{remark}
The vector space of $3 \times 3$ magic squares $W^{\perp}$ that controls the proof is exactly the space parametrized by Lucas~\cite{lucas_1891_sur} in the introduction. Indeed, in the basis above, a general magic square is
\[
z\,J - y\,N_{1} + x\,N_{2}
  \;=\;
  \begin{pmatrix}
    z - y & x + y + z & z - x \\
    z + y - x & z & z + x - y \\
    x + z & z - x - y & y + z
  \end{pmatrix},
\]
so $(x, y, z)$ are linear coordinates on $W^{\perp}$, with $z$ the center. Reading each entry of the Lucas square as a linear form in $(x, y, z)$ recovers the nine points of the Gale transform, with the labelling described in the proof of Theorem \ref{thm:main3}.
\end{remark}

\subsection{Trivial and Singular Points}
\label{subsec:trivialpoints}
This description of the automorphism group can be used to gain a simpler understanding of some rational points and the singular points on $V$. For example, $V$ has 256 obvious rational points corresponding to $[\pm1, \pm1, \pm1, \pm1, \pm1, \pm1, \pm1, \pm1, \pm1]$, all yielding the trivial magic square of squares
\[
\magicsquarethree{1}{1}{1}{1}{1}{1}{1}{1}{1}
\]
These 256 points are the orbit of the point $[1, 1, 1, 1, 1, 1, 1, 1, 1]$ under the automorphism group.

Secondly, the 256 singular points of $V$ are the unions of the orbits of the points 
\begin{align*}
p_1 &{} = [0, 2, \sqrt{2}, 2, \sqrt{2}, 0, \sqrt{2}, 0, 2],\\
p_2 &{} = [i, 0, 1, \sqrt{2}, 0, \sqrt{-2}, i, 0, 1],\\ 
p_3 &{} = [i, 1, 0, 1, 0, i, 0, i, 1], \\
\end{align*}
under the action of the automorphism group. The orbit of $p_1$ is of order 128, while the orbits of $p_2$ and $p_3$ are of order 64. These correspond to the squares
\[
\magicsquarethree{0}{4}{2}{4}{2}{0}{2}{0}{4}
\hspace{10pt} \magicsquarethree{-1}{\;0\;\;}{1}{2}{\;0\;\;}{-2}{-1}{\;0\;\;}{1} \hspace{10pt} \magicsquarethree{-1}{1}{0}{1}{0}{-1}{0}{-1}{1}
\]
Later on, we will see that the automorphism group will aid in generating many divisors on $V$.

\section{Hodge Numbers and Associated Invariants}

In this section, we compute the basic topological invariants and Hodge diamond of the resolution $\widetilde{V}$ of the variety of $3 \times 3$ magic squares of squares. Recall that for a K\"ahler manifold $X$, its singular cohomology decomposes as
\[
H^i(X, \mathbb{C}) \cong \bigoplus_{ p + q = i} H^{q}(X, \Omega^{p}_{X/\mathbb{C}})
\]
for $0 \leq i \leq 2n$. We will use $h^{p, q}$ to denote $\dim_{\mathbb{C}} H^{q}(X, \Omega^{p}_{X/\mathbb{C}})$, and call the collection of $h^{p, q}$'s the Hodge numbers of $X$. In the case that $X$ is the analytification of a variety $Y$ over a number field $K$, GAGA~\cite{serre_1956_gomtrie} implies that $h^{p, q} = \dim_{K}H^q(Y, \Omega^{p}_{Y/K})$. Letting $b_i$ denote the $i$th Betti number of $X$, we get the formula
\[
b_i = \sum_{p + q = i} h^{p, q}.
\]
For more details, see \cite{griffiths2014principles}.

In the case of $V$, since it has only ordinary double points, we can show that it appears as the central fiber of a flat family of complete intersections whose general member is smooth.

\begin{lemma}\label{lem:smoothing}
Write $V = V(q_1, \dots, q_6) \subseteq \mathbb{P}^8$. There exist quadrics $f_1,\dots,f_6$ on $\P^8$ such that the family
\[
  \mathcal{X} = \bigg\{(x,t)\in\P^8\times\mathbb{A}^1 \  \bigg| \ 
  q_i(x)+tf_i(x)=0,\ i=1,\cdots,6 \bigg\}
\]
has smooth total space in an analytic neighborhood of $\mathcal X_0=V$ over $\mathbb{C}$, and smooth fiber $\mathcal X_t$ for all $t\neq 0$ in some analytic disc $\Delta$ containing $0 \in \mathbb{A}^1$.
\end{lemma}
 
\begin{proof}
Let $F_i=q_i+tf_i$. At a point $(x,0)$ with $x\in V$ smooth, the Jacobian $J(x) := \bigl[\partial F_i/\partial x_j\bigr]_{i, j}$ already has rank $6$, so $\mathcal X$ is smooth at $(x,0)$ for any choice of $f_i$ not all vanishing at $x$. At a node $p$ of $V$, $J(p)$ has rank $5$. Let $\Lambda \subset\mathbb C^6$ denote its image, a hyperplane. The Jacobian of $\mathcal X$ at $(p,0)$ is $J$ together with the extra column $\bigl(f_1(p),\dots,f_6(p)\bigr)^{\top}$ coming from $\partial/\partial t$. So, it has rank $6$, i.e., $\mathcal X$ is smooth at $(p,0)$, when
\begin{equation}\label{eq:node-transversality}
  \bigl(f_1(p),\dots,f_6(p)\bigr)\;\notin\;\Lambda .
\end{equation}
For a fixed $p\neq 0$, evaluation $f\mapsto f(p)$ is a nonzero linear functional on quadrics, so $(f_1,\dots,f_6)\mapsto \bigl(f_i(p)\bigr)$ is surjective onto $\mathbb C^6$. This implies that \eqref{eq:node-transversality} is a nonempty Zariski-open condition on the $6$-tuple of $f_i$, and the intersection of these conditions over the $256$ nodes is again nonempty and open. Since a general 6-tuple of quadrics in $\mathbb{P}^8$ yields a smooth complete intersection surface, we can choose $(f_1, \dots, f_6)$ so that the pencil $\{q+tf\}_t$ meets the locus of smooth complete intersections. So, $\mathcal X_t$ is
smooth for general $t$. Shrinking to a disc $\Delta$ finishes the proof.
\end{proof}
By a theorem of Atiyah~\cite{atiyah_1958_on} on the topology of degenerations with ordinary double points, Lemma \ref{lem:smoothing} implies that $\widetilde{V}$ is diffeomorphic to a nearby smooth fiber. Hence, $\widetilde{V}$ has the same Betti numbers and signature as the general complete intersection of type $(2, 2, 2, 2, 2, 2)$. For a compact K\"ahler surface, the Hodge numbers are determined by the Betti numbers and the signature $\tau$ of the intersection form. So, they are invariant under diffeomorphism. The Chern numbers of a K\"ahler surface can be similarly determined, so they are also invariant under diffeomorphism. Thus, the Hodge and Chern numbers of $\widetilde{V}^{an}$ are equal to the Hodge and Chern numbers of the generic complete intersection surface of type $(2, 2, 2, 2, 2, 2)$.

The cohomology of complete intersections is well understood \cite{deligne2006cohomologie, fhirzebruch_1978_topological}; we provide a recap of the computation here. For any complete intersection $X$ of dimension $d$ and for any $0 < i< d$, we have that $h^i(X, \mathscr{O}_{X}) = 0$, so $q(\widetilde{V}^{an}) = h^{0, 1}(\widetilde{V}^{an}) = h^{1, 0}(\widetilde{V}^{an})$ = 0.

It remains to calculate $h^{1, 1}$ and $p_g = h^{0, 2}$. However, for a general complete intersection $V_n$ of 6 quadrics in $\mathbb{P}^{n+6}$, denoting 

\[
\chi^{p}(V_n) = \sum_{q \geq 0} (-1)^{q}h^{p, q}(V_n)
\]
for the $p$th holomorphic Euler characteristic, we have the formula 
\[
\chi^p(V_n) = 
\begin{cases}
 (-1)^{n-p}h^{p, n-p}(V_n) + (-1)^p, \quad &{} 2p \neq n \\
 (-1)^{p}h^{p, p}(V_n), \quad &{} 2p = n
\end{cases}
\]
by Hirzebruch--Riemann--Roch. We can obtain these by the formula
\[  
\sum_{n \geq 0}\sum_{p \geq 0} \chi^{p}(V_n)y^{p}z^{n+6} = \frac{1}{(1-z)(1+zy)}\prod_{i = 1}^{6} \frac{(1+zy)^{2} - (1-z)^{2}}{(1+zy)^{2} + y(1-z)^{2}}
\]
given in \cite{fhirzebruch_1978_topological}. So, to compute the Hodge numbers for $\widetilde{V}^{an}$, we only need to compute the coefficients of the $z^8$, $yz^8$, and $y^2z^8$ terms. Using an algorithm for this \cite{belmans_2018_hodge}, we compute the Hodge diamond
\[
\begin{array}{ccccc}
 &  & 1 &  & \\
 & 0 &  & 0 & \\
111 &  & 544 &  & 111\\
 & 0 &  & 0 & \\
 &  & 1 &  &
\end{array}
\]
This yields the Betti numbers $b^0$ = 1, $b^1$ = 0, $b^2$ = 766, $b^3$ = 0, $b^4$ = 1, and topological Euler characteristic $\chi_{top}(\widetilde{V}^{an})$ = 768. By Poincar\'e duality and the Lefschetz hyperplane theorem forcing $H^n(\widetilde{V}^{an}, \mathbb{Z}) \cong H^n(\mathbb{P}^8, \mathbb{Z})$ for all $n \neq 2$, which is torsion-free, and the universal coefficient theorem for $H^2(\widetilde{V}^{an}, \mathbb{Z})$, the singular cohomology of $\widetilde{V}^{an}$ is torsion-free. 

For the Chern numbers, we use the following result:
\begin{prop}[\cite{bruin_2022_explicit}]  
For a complete intersection $X$ of $n-2$ quadrics in $\mathbb{P}^n$, its Chern numbers are given by the formulas
\begin{align*}
c_{1}^2(X) = &{\,} (n-5)^{2}2^{n-2},\\
c_2(X) = \chi_{top}(X) = &{\,} (n^2 - 7n + 16)2^{n-3}.
\end{align*}
\end{prop}
\noindent Here, $n = 8$, so for $\widetilde{V}^{an}$ we recover the Chern numbers $c_{1}^2(\widetilde{V}^{an}) = 576$, $c_2(\widetilde{V}^{an}) = 768$.

Next, we determine the structure of $H^2(\widetilde{V}^{an}, \mathbb{Z})$ as a lattice. The lattice is unimodular by Poincar\'e duality, since $H^2(\widetilde{V}^{\mathrm{an}}, \mathbb{Z})$ is torsion-free. To determine its parity, recall that $\widetilde{V}^{an}$ is diffeomorphic to a smooth complete intersection $X$ of type $(2,2,2,2,2,2)$ in $\mathbb{P}^8$, and that the isomorphism class of the lattice with its intersection form is invariant under diffeomorphism. So, it suffices to compute the parity on $X$, where it is determined by the second Stiefel--Whitney class.

Let $[H] \in H^2(X, \mathbb{Z})$ be the first Chern class of the hyperplane section. Since $X$ is a complex manifold, its Stiefel--Whitney classes are the mod 2 reductions of its Chern classes \cite{milnor1974characteristic}. Since $[K_X] = 3[H]$ by adjunction, we have
\[
w_2(X) = c_1(X) \ \mathrm{mod}  \ 2 = -3h \ \mathrm{mod} \ 2 = h \ \mathrm{mod} \ 2
\]
We claim that $h$ is nonzero mod 2. Indeed, the pair $(\mathbb{P}^8, X)$ is 2-connected by the Lefschetz hyperplane theorem, so $H^2(\mathbb{P}^8,  X; \mathbb{Z}/2\mathbb{Z}) = 0$ by the universal coefficient theorem. The long exact sequence of the pair yields an injection $H^2(\mathbb{P}^8, \mathbb{Z}/2\mathbb{Z}) \hookrightarrow H^2(X, \mathbb{Z}/2\mathbb{Z})$. The hyperplane class generates the former, so its restriction $h$ is nonzero in $H^2(X, \mathbb{Z}/2\mathbb{Z})$. Hence, $w_2(X) \neq 0$.

Since $X$ is orientable, the second Wu class of $X$ coincides with $w_2(X)$. Hence, Wu's formula gives $x \cdot x = w_2(X) \cdot x$ for all $x \in H^2(X, \mathbb{Z}/2\mathbb{Z})$. The mod 2 cup product pairing on $H^2(X, \mathbb{Z}/2\mathbb{Z})$ is nondegenerate by Poincar\'e duality, so there is a class $x$ with $w_2(X) \cdot x = 1$, and thus with $x \cdot x = 1$. Moreover, $H^3(X, \mathbb{Z}) = 0$, so the reduction map $H^2(X, \mathbb{Z}) \to H^2(X, \mathbb{Z}/2\mathbb{Z})$ is surjective and $x$ lifts to an integral class of odd self-intersection. 

Any odd self-intersection forces the intersection form to be odd. By the Hodge index theorem, the signature is $\tau = b_2^+ - b_2^- = -320$, so $(b_2^+, b_2^-) = (223, 543)$. Hence, by the classification of indefinite odd unimodular lattices \cite{serre2012course}, we have $H^2(\widetilde{V}^{an}, \mathbb{Z}) \cong \langle 1 \rangle^{\oplus 223} \oplus \langle -1\rangle^{\oplus 543}$.

Finally, we remark that since $q(\widetilde{V}) = 0$, the Albanese variety of $\widetilde{V}$, $\mathrm{Alb}(\widetilde{V})$, is a point. Hence, $\widetilde{V}$ admits no interesting morphisms to abelian varieties.

\section{Special Hyperplane Sections of $V$}
\label{sec:divisors}
In this section, we produce $416$ explicit divisor classes on $\widetilde{V}$. These will form the first major collection of generators used in the Picard rank computation of Section \ref{sec:PicardLattice}. We obtain them by taking irreducible components of split hyperplane sections defined over various number fields, many of which arise naturally from the Diophantine problem itself.  
First is the locus of squares with non-distinct entries (\ref{subsec:nondistinct}), whose components contain some rational curves, as well as high genus curves corresponding to magic squares of squares with five or seven distinct entries. Second is the coordinate hyperplane sections (\ref{subsec:coordinate_hyperplane_sections}), which give a link to the congruent number elliptic curve $y^2 = x^3 - d^2x$ and yield new families of magic squares of squares over biquadratic fields. The remaining hyperplane sections, defined over number fields such as $\Q(i, \sqrt{3})$ and $\Q(\sqrt{-2})$, are more procedural and handled in Section~\ref{sec:numberfields}. Finally, in Appendix~\ref{subsec:lines-conics}, we use a Schubert-cell stratification of the projective Grassmannian $\mathbb{G}(1,8)$, following Elsenhans--Jahnel~\cite{elsenhans2008k3}, to show that $V$ contains no lines.

\subsection{The Locus of Squares with Non-Distinct Entries}
\label{subsec:nondistinct}
Consider the closed subscheme $Z \subset V$ consisting of squares with non-distinct entries.  We refer to $Z$ as the \linedef{nondistinct locus} and to its complement $U = V \smallsetminus Z$ the \linedef{distinct locus}. The nondistinct locus is the union of hyperplane sections defined by equating elements around the square. A pattern emerges when decomposing these hyperplane intersections into irreducible components. 

Intersecting $V$ with the hyperplane $A - B = 0$ illustrates the general behavior. Doing so sets $A^2 = B^2$, yielding the equations 
\begin{align*}
    A^2 + G^2 - 2M^2 = &{\,} 0,\\
    A^2 + H^2 - 2M^2 = &{\,} 0,
\end{align*}
implying that $G^2 = H^2$. Hence, the hyperplane section $A - B = 0$ splits into two irreducible components determined by a choice of sign in
\begin{align*}
    A - B = &{\,} 0, \\
    G \pm H =  &{\,}0.
\end{align*}
The interpretation in terms of magic squares is a square with 7 distinct entries
\[
\magicsquarethree{A^2}{A^2}{C^2}{D^2}{M^2}{E^2}{F^2}{H^2}{H^2}
\]
Intersecting $V$ with hyperplanes equating entries of the square that are adjacent horizontally or vertically yields analogous decompositions into two irreducible components, since these squares comprise the orbit of the hyperplane section $A = B$ by $\mathrm{Aut}(V)$. 

However, if we intersect $V$ with the hyperplane $A - C = 0$, we get eight irreducible components indexed by the choices of signs in
\[
A - C =  D \pm E = M \pm E = F \pm H =  0.
\]
This corresponds to magic squares of the form
\[
\magicsquarethree{A^2}{B^2}{A^2}{M^2}{M^2}{M^2}{F^2}{G^2}{F^2}
\]
Intersecting $V$ with the hyperplane equating any two same-row or same-column corner entries yields analogous decompositions into eight irreducible components. 

When 2 entries across the left diagonal are set equal (i.e., intersecting with the hyperplanes $A = \pm M$ or $H = \pm M$, which yield the same components), we are left with 64 irreducible components corresponding to a square with 3 distinct entries
\[
\magicsquarethree{A^2}{B^2}{C^2}{C^2}{A^2}{B^2}{B^2}{C^2}{A^2}
\]
with the relation $B^2 + C^2 = 2A^2$, i.e., the conic parametrizing squares in arithmetic progression. This is a smooth conic with a rational point $[1:1:1]$, so it is a rational curve on $V$. These 64 rational curves realize magic squares of squares with 3 distinct entries, such as 
\[
\magicsquarethree{13^2}{7^2}{17^2}{17^2}{13^2}{7^2}{7^2}{17^2}{13^2}
\]
A similar computation can be carried out by setting any 2 entries across the right diagonal equal (i.e. $C = \pm M$ or $F = \pm M$), yielding 64 more rational curves.

We also note that $V$'s singular points are entirely contained within this locus and the five distinct entry loci, meaning that $U$ is smooth. The hyperplane sections where the splittings will occur are:
\begin{center}
   \begin{tabular}{c|c|c}
      3 Distinct Entries  &  5 Distinct Entries & 7 Distinct Entries\\
      \hline
      $A \pm M$  & $A \pm C$ & $A \pm B$ \\
      $C \pm M$ & $A \pm F$ & $B \pm C$\\
      & & $A \pm D$\\
      &  & $C \pm E$\\
      
   \end{tabular}
\end{center}
Indeed, other hyperplane sections either do not split further, yielding a generic hyperplane section, or are covered by those listed; for example, $B = \pm M$ forces $G^2 = M^2$, which in turn forces $A^2 = F^2$. 

We will denote the three distinct entry locus $Z_3$, the five distinct entry locus $Z_5$, and the seven distinct entry locus $Z_7$. In summary, we have the following:

\begin{prop}
\label{components}
The nondistinct locus $Z \subset V$ decomposes into $Z_3 \cup Z_5 \cup Z_7$, where each component $Z_i$ generically corresponds to squares with $i$ distinct entries. All irreducible components of $Z$ are smooth geometrically connected curves defined over $\mathbb{Q}$ whose degree and genera are given by the following table: 

\begin{center}
   \begin{tabular}{c|c|c|c}
     &  $Z_3$  &  $Z_5$ & $Z_7$  \\
     \hline \#Components & 128 & 32 & 16 \\
     Genus & 0 & 5 & 49 \\
     Degree & 2 & 8 & 32
   \end{tabular}
\end{center}
In particular, $Z$ has 176 irreducible components in total.
\end{prop}

\begin{proof}
For each of the above components $C$ defined over $\mathbb{Q}$, we verify using Magma~\cite{bosma1997magma} that $\dim_{\mathbb{Q}} H^0(C, \mathscr{O}_C) = 1$. This shows that they are geometrically connected and irreducible.  We also show using Magma that each component is smooth.
\end{proof}

One can check that all components of $Z$ meet in either the 256 trivial rational points or the 256 singular points of $V$.

\begin{remark}
The fact that the components of $Z_5$ and $Z_7$ are of high genus is consistent, via Faltings theorem, with older results to the effect that there are no integer $3 \times 3$ magic squares of squares with 5 or 7 distinct entries, see \cite{boyer_multimagiecom}.    
\end{remark}

All other curves presented in this section have been verified to be geometrically connected and irreducible in a similar fashion as in the proof of Proposition~\ref{components}, so these comments will be omitted in the sequel. 

\subsection{Coordinate Hyperplane Sections}
\label{subsec:coordinate_hyperplane_sections}

Another natural family of hyperplane sections to consider are the coordinate hyperplanes. We start with the case $M = 0$. This corresponds to the square
\[
\magicsquarethree{A^2}{B^2}{C^2}{D^2}{0}{-D^2}{-C^2}{-B^2}{-A^2}
\]
with equations 
\begin{align}
\label{eq:conics_m=0}
\begin{split}
    A^2+ D^2 =&{\,} C^2,\\
     B^2 +2C^2 =&{\,} D^2.
\end{split}
\end{align}
This decomposes into 8 irreducible components over $\mathbb{Q}$, which actually decompose further over $\Qbar$. 

\begin{prop}
The $M=0$ hyperplane section decomposes, over $\Q(i)$, into the union of 16 smooth elliptic curves all isomorphic to $y^2 = x^3 - x$.     
\end{prop}
\begin{proof}
Using Magma~\cite{bosma1997magma}, we check that each of the 8 irreducible components over $\Q$ decomposes over $\Q(i)$ into a disjoint union of two elliptic curves, each of which is isomorphic to $y^2 = x^3 - x$.  
\end{proof}

This yields the following new connection with the congruent number problem.

\begin{prop}
   For $d$ a congruent number, the distinct locus $U$ has infinitely many points with $M = 0$ over $\mathbb{Q}(i, \sqrt{d})$. 
\end{prop}
\begin{proof}
For any squarefree integer $d$, the elliptic curves $y^2 = x^3 - d^2x$ are in the same twist family, i.e., $y^2 = x^3 - x$ is isomorphic to $y^2 = x^3 - d^2x$ over $\mathbb{Q}(\sqrt{d})$. The elliptic curve $y^2 = x^3 - d^2x$ has positive rank over $\Q$ if and only if $d$ is a congruent number. Hence, each irreducible component of the $M=0$ locus over $\Q(i)$ has infinitely many rational points over $\Q(i,\sqrt{d})$ whenever $d$ is a congruent number, and all but finitely many of these points cannot be in the nondistinct locus as these curves are not themselves contained in the nondistinct locus.
\end{proof}

We remark that a different connection between $3 \times 3$ magic squares of squares and the congruent number problem has been explored by Robertson~\cite{robertson_1996_magic}, who noted that a magic square of squares can be given by a sequence of points on a congruent number elliptic curve in arithmetic progression. 

\begin{remark}
This result also gives a method of proving that a $3 \times 3$ magic square of squares with $M=0$ and with coefficients in $\mathbb{Z}[i]$ is impossible.  Indeed, the elliptic curve $y^2 = x^3 - x$ has Mordell--Weil group structure $\mathbb{Z}/2\mathbb{Z} \oplus \mathbb{Z}/4\mathbb{Z}$ over $\mathbb{Q}(i)$ \cite{lmfdb}. This corresponds to the orbit of the square
\[
\magicsquarethree{0}{1}{-1}{-1}{0}{1}{1}{-1}{0}
\]
under the automorphism group, which is contained in $Z_3$. This was previously proved by Sallows in a more elementary way, see~\cite{boyer_multimagiecom}. 
\end{remark}

We can construct explicit families of magic squares of squares over biquadratic fields by leveraging the geometry of the elliptic curve components.  In particular, we consider one of the quadrics in the defining equations \eqref{eq:conics_m=0} and project to the $\mathbb{P}^2$ that it spans, which presents each of the elliptic curves as a double cover of a smooth plane conic. 

We start with the case of $A^2 +D^2 = C^2$. This conic is the rational curve parametrizing Pythagorean triples $(A,D,C)$. This, in turn, determines $B = \sqrt{D^2 - 2C^2}$ up to sign, which is contained in an extension of degree at most 2.  The rest of the square is then determined up to signs by adjoining $i$.  Explicitly, by using Euclid's parametrization of Pythagorean triples, we can write
\[
A = m^2 - n^2, D = 2mn, C = m^2 + n^2.
\]
for distinct integers $m$ and $n$. Substituting $B^2 = D^2 - 2C^2$, we then get
\[
B = \sqrt{-2(m^4 + n^4)},
\]
yielding a square
\[
\magicsquarethree{(m^2 - n^2)^2}{-(\sqrt{2(m^4 + n^4)})^2}{ (m^2 + n^2)^2}{(2mn)^2}{0}{-(2mn)^2}{-(m^2 + n^2)^2}{(\sqrt{2(m^4 + n^4) })^2}{-(m^2 - n^2)^2}
\]
over the biquadratic field $\mathbb{Q}(i, \sqrt{2(m^4 + n^4)})$. An example from this family over the number field of smallest discriminant is 
\[
\magicsquarethree{9}{-34}{25}{16}{0}{-16}{-25}{34}{-9}
\]
defined over $\mathbb{Q}(i, \sqrt{34})$, by letting $(m,n) = (2,1)$, corresponding to the Pythagorean triple $(A,D,C) = (3,4,5)$. In fact, this is a different family from Bremner's family mentioned in the introduction 
\[
\magicsquarethree{(q^2 + 1)^2}{-(q^2 + 2q - 1)^2}{(2\sqrt{q^3 - q})^2}{-(q^2 - 2q - 1)^2}{0}{(q^2 - 2q - 1)^2}{-(2\sqrt{q^3 - q})^2}{(q^2 + 2q - 1)^2}{-(q^2 + 1)^2}
\]
over $\mathbb{Q}(i, \sqrt{q^3 - q})$, for a positive integer $q$.

For the other case, the plane curve $B^2 +2C^2= D^2$ is parametrized by
\[
B = m^2 - 2n^2, C = 2mn, D = m^2 + 2n^2
\]
for nonzero integers $m$ and $n$. This determines $A = \sqrt{C^2 - D^2}$ up to sign, specifically,
\[
A = i\sqrt{m^4 + 4n^4}
\]
and the family of squares
\[
\magicsquarethree{-(\sqrt{m^4 + 4n^4})^2}{(m^2 - 2n^2)^2}{(2mn)^2}{(m^2 + 2n^2)^2}{0}{-(m^2 + 2n^2)^2}{-(2mn)^2}{ -(m^2 - 2n^2)^2}{(\sqrt{m^4 + 4n^4 })^2}
\]
over the number field $\mathbb{Q}(i, \sqrt{m^4 + 4n^4})$. An example over the number field of smallest discriminant is
\[
\magicsquarethree{-5}{1}{4}{9}{0}{-9}{-4}{-1}{5}
\]
over $\mathbb{Q}(i, \sqrt{5})$, defined by $(m,n)=(1,1)$, corresponding to the sequence of squares $(B,C,D)$.  In fact, this discriminant is smaller than is achievable for Bremner's family.

Next, we consider the case of setting corner elements to $0$. These curves split over the biquadratic field $\mathbb{Q}(\sqrt{2}, i)$. We will first examine the hyperplane section $A = 0$, as the hyperplane sections $C = 0$, $F = 0$, and $H = 0$ are all in its orbit under the automorphism group. Setting $A = 0$ corresponds to the square
\[
\magicsquarethree{0}{B^2}{C^2}{D^2}{M^2}{-G^2}{F^2}{G^2}{2M^2}
\]
forcing the relations
\begin{align*}
    H =  &{\,}\pm \sqrt{2}M, \\
    E  = &{\,}\pm iG.
\end{align*}
The $A=0$ hyperplane section decomposes over $\mathbb{Q}(i, \sqrt{2})$ into a union of 4 smooth geometrically connected curves of genus 17 and degree 16, each subject to the above relations along with
\begin{align*}
    C^2 - G^2 = &{\,} M^2, &{} B^2 + G^2 = &{\,} 2M^2, \\
    F^2 + G^2 = &{\,}M^2,  &{}  D^2 - G^2 = &{\,} 2M^2.
\end{align*}

Finally, we consider the case of setting elements in the middle of an outer row or column to 0. These curves split over the quadratic field $\mathbb{Q}(\sqrt{2})$. We will only examine the hyperplane section $B = 0$, as the hyperplane sections $D = 0$, $E = 0$, and $G = 0$ are all in its orbit under the automorphism group. Setting $B = 0$, which corresponds to the square
\[
\magicsquarethree{A^2}{0}{C^2}{2H^2}{M^2}{2F^2}{F^2}{2M^2}{H^2}
\]
forces the linear relations
\begin{align*}
    D &{} = \pm \sqrt{2}H, \\
    G &{} = \pm \sqrt{2}M, \\
    E &{} = \pm \sqrt{2}F
\end{align*}
and the quadratic equations
\begin{align*}
    A^2 + H^2 &{} = G^2, \\
    G^2 + 2H^2 &{} = 2C^2, \\
    2F^2 + 2H^2 &{} = G^2, 
\end{align*}
yielding 8 smooth geometrically connected curves of genus 5 and degree 8. In summary, we have the following.

\begin{prop}
Let $Z' \subset V$ denote the locus given by coordinate hyperplane sections. Then $Z'$ decomposes into $Z_M \cup Z_{\mathrm{middle}} \cup Z_{\mathrm{corner}}$, where $Z_M$ corresponds to squares with $M = 0$, $Z_{\mathrm{middle}}$ corresponds to squares with the middle entry of an outer row or column 0, and $Z_{\mathrm{corner}}$ corresponds to squares with corner entries 0. All irreducible components of $Z'$ are smooth geometrically connected curves defined over $\mathbb{Q}(i, \sqrt{2})$ whose degree and genera are given by the following table:
    \begin{center}
        \begin{tabular}{c|c|c|c}
        & $Z_M$ & $Z_{\mathrm{middle}}$ & $Z_{\mathrm{corner}}$ \\
        \hline
        \#Components & 16 & 32 & 16 \\
        Genus & 1 & 5 & 17 \\
        Degree & 4 & 8 & 16 
        \end{tabular}
    \end{center}
    In particular, $Z'$ has 64 geometrically irreducible components in total.
\end{prop}

\subsection{Hyperplane Sections over Number Fields}
\label{sec:numberfields}
To complete our collection of curves in Section 3, we take many split hyperplane sections over number fields. The next most natural family of hyperplane sections are of the form $A = \pm\sqrt{d} M$ and $B = \pm\sqrt{d} M$ for a square-free positive integer $d$. Indeed, setting such a relation will yield squares of the form
\[
\magicsquarethree{dM^2}{B^2}{C^2}{D^2}{M^2}{E^2}{F^2}{G^2}{(2-d)M^2}
\]
that will give a split hyperplane section over the number field $\mathbb{Q}(\sqrt{d}, \sqrt{2 - d})$ due to also inducing $H = \pm \sqrt{2 - d}M$. As seen in the previous section, the case $d = 2$ is degenerate, yielding $H = 0$ and further splitting of the hyperplane section. The case $d = 3$ is examined below. However, for $d > 3$, these hyperplane sections fail to split further and become linearly dependent; their intersection matrix is rank 1. Thus, we cannot use this method to produce the rest of the Picard lattice. 

However, there are still a few families of other split hyperplane sections over number fields of small degrees (namely $\mathbb{Q}(i, \sqrt3)$ and $\mathbb{Q}(\sqrt{-2})$) that yield new linearly independent classes. Their generation is more procedural than our work in Sections \ref{subsec:nondistinct} and \ref{subsec:coordinate_hyperplane_sections}, as these hyperplane sections do not unearth any intriguing new Diophantine connections. Similarly to the proof of Proposition \ref{components}, all curves were verified to be geometrically connected and irreducible in Magma. 

\subsubsection{Over {$\mathbb{Q}(i, \sqrt{3})$}}

Many hyperplane sections taken over $\mathbb{Q}(i)$ need to be lifted to $\mathbb{Q}(i, \sqrt{3})$ for their components to be fully seen. We will outline these here. Up to automorphism, there are two types of hyperplane sections we will need to consider:

Firstly, $B = \pm iC$ induces degree 8, genus 5 curves which correspond to the square
\[
\magicsquarethree{A^2}{B^2}{-B^2}{D^2}{\frac{1}{3}A^2}{E^2}{-D^2}{G^2}{-\frac{1}{3}A^2}
\]
There are 16 of these upon accounting for all possible signs. These can be permuted in 4 ways via the action of the dihedral group, giving us 64 divisors of this type.

Secondly, $A = \pm iC$ induces degree 16, genus 17 curves which correspond to the square
\[
\magicsquarethree{A^2}{-3G^2}{-A^2}{D^2}{-G^2}{E^2}{F^2}{G^2}{H^2}
\]
There are 8 of these upon accounting for all possible signs. These can be permuted in 4 ways via the action of the dihedral group as it is invariant under the reflection, giving us 32 divisors of this type and 96 in total.

\subsubsection{Over {$\mathbb{Q}(\sqrt{-2})$}}

Up to automorphism, there are two classes of hyperplane sections over $\mathbb{Q}(\sqrt{-2})$. We outline them here:

Firstly, $A = \pm \sqrt{-2}M$ induces degree 32, genus 49 curves corresponding to the square
\[
\magicsquarethree{A^2}{B^2}{C^2}{D^2}{-\frac{1}{2}A^2}{E^2}{F^2}{G^2}{-2A^2}
\]
There are 4 of these upon accounting for all possible signs. These can be permuted in 4 ways via the action of the dihedral group, giving us 16 divisors of this type.

Secondly, $B = \pm \sqrt{-2}M$ induces degree 8, genus 5 curves which correspond to the square
\[
\magicsquarethree{A^2}{B^2}{C^2}{D^2}{-\frac{1}{2}B^2}{E^2}{-\frac{1}{2}D^2}{-2B^2}{-\frac{1}{2}E^2}
\]
There are 16 of these upon accounting for all possible signs. These can be permuted in 4 ways via the action of the dihedral group, as it is invariant under the reflection, giving us 64 divisors of this type and 80 in total. 

In summary, we have the following:

\begin{prop}
    Let $Z'' \subseteq V$ denote the locus given by hyperplane sections over $\mathbb{Q}(i, \sqrt3)$ and $\mathbb{Q}(\sqrt{-2})$. Then $Z''$ decomposes as $Z_{A, C} \cup Z_{B, C} \cup Z_{A, M} \cup Z_{B, M}$, where $Z_{A, C}$ and $Z_{B, C}$ correspond to squares with $A^2 = -C^2$ and $B^2 = -C^2$ respectively, and $Z_{A, M}$ and $Z_{B, M}$ correspond to squares with $A^2 = -2M^2$ and $B^2 = -2M^2$ respectively. All irreducible components of $Z''$ are smooth geometrically connected curves defined over $\mathbb{Q}(i, \sqrt3)$ or $\mathbb{Q}(\sqrt{-2})$, whose degree and genera are given by the following table:

\begin{center}
        \begin{tabular}{c|c|c|c|c}
        & $Z_{A, C}$ & $Z_{B, C}$ & $Z_{A, M}$ & $Z_{B, M}$ \\
        \hline
        \#Components & 32 & 64 & 16 & 64\\
        Genus & 17 & 5 & 49 & 5 \\
        Degree & 16 & 8 & 32 & 8
        \end{tabular}
    \end{center}
    In particular, $Z''$ has 176 geometrically irreducible components in total.
\end{prop}

\section{Magic Surfaces}
\label{sec:magic-surfaces}
A related problem to finding a $3 \times 3$ magic square of squares is finding a magic square with as many square entries as possible \cite{boyer_multimagiecom}. For instance, there is a monetary bounty placed on producing a square of 7 squares different from the only known example
\begin{equation}
\label{eq:7squares}
\magicsquarethree{373^2}{289^2}{565^2}{360721}{425^2}{23^2}{205^2}{527^2}{222121}
\end{equation}
In this section, we study a number of projections from $V$ to del Pezzo and K3 surfaces, motivated by studying magic squares with a prescribed number of square entries. In the process, we will construct 532 new divisors on $V$ as preimages of lines on these surfaces, aiding in the proof of Theorem~\ref{thm:picard-rank}. We will use the following definition.
\begin{defi}
A surface $S$ defined over $\mathbb{Q}$ is \linedef{magic} if it admits a dominant rational map ${V \dashrightarrow S}$.
\end{defi}

The magic surfaces we will consider arise as resolutions of images of coordinate projections from $V$ to various projective spaces. In particular, we will consider resolutions which are del Pezzo surfaces and K3 surfaces of high Picard rank.

\subsection{Magic del Pezzo Surfaces}
\label{subsec:magicdPs}
We begin with magic squares with 5 square entries, which we study by considering the restriction to $V$ of coordinate projections $\mathbb{P}^8 \dashrightarrow \mathbb{P}^4$.  Since the center of each projection is disjoint from $V$, we obtain morphisms $V \to \mathbb{P}^4$.  

We will show that the image of each projection is an intersection of 2 quadrics in $\mathbb{P}^4$, i.e., a magic quartic del Pezzo surface, which is possibly singular.  We further classify these quartic del Pezzo surfaces over $\Qbar$ via invariant theory.

We now recall the classical description of the coarse moduli space of (mildly singular) quartic del Pezzo surfaces, see \cite[Section~8.6.1]{dolgachev2012classical}. To each intersection of two quadrics $Q_1 \cap Q_2 \subset \P^4$ there is an associated pencil of quadrics, whose total space is the bidegree $(1,2)$ hypersurface $Q \subset \P^1 \times \P^4$ defined by $q = u q_1 + v q_2$, where $Q_i=V(q_i)$ and $(u:v)$ is a set of homogeneous coordinates on $\P^1$.  The projection $Q \to \P^1$ is a quadric threefold bundle whose discriminant has degree 5, cut out by the determinant of the Gram matrix of $q$ over $\Q[u,v]$, which is a symmetric $5\times 5$ matrix of linear forms on $\P^1$.  The vector of geometric elementary divisor data of this matrix, which encodes the multiplicities of geometric points of the discriminant and the coranks of the associated singular quadrics, is called the \linedef{Segre symbol}.  The complete intersection is smooth if and only if the discriminant is reduced and it has at most $A_1$-singularities if all points of the discriminant have multiplicity $\leq 2$.  Specifically, the complete intersection: is smooth if and only if the Segre symbol is $[1,1,1,1,1]$; has two $A_1$-singularities connected by a line if and only if the discriminant has a single point of multiplicity 2 and the associated quadric has corank 2 there, equivalently, the Segre symbol is $[(1,1),1,1,1]$; it has four $A_1$-singularities if and only if the discriminant has a pair of points of multiplicity 2 and the associated quadrics have corank 2 there, equivalently, the Segre symbol is $[(1,1),(1,1),1]$.  A classical result of Weierstrass and Segre is that the Segre symbol, together with the $\mathrm{PGL}_2(\Qbar)$-orbit of the multiset of the five geometric points of the discriminant, counted with multiplicity (which we call the \linedef{discriminant configuration}), completely determines the $\Qbar$-isomorphism class of a (possibly singular) quartic del Pezzo surface, see \cite[Section~2]{avritzer2000pencils} or \cite[Theorem~8.6.3]{dolgachev2012classical} for modern accounts of this fact.  

We can reduce the 126 coordinate linear projections $V \to \P^4$ by only considering orbits under the automorphism group.  These are in bijection with orbits of 5 entries of a $3 \times 3$ grid under the natural action of the dihedral group. Coordinate projections corresponding to the same orbit will yield restrictions to $V$ that differ by automorphisms.  By Burnside's lemma, there are 23 orbits, with representatives given in Figure~\ref{fig:5}.  

\begin{figure}
\noindent 
\begin{minipage}[t]{0.22\textwidth}
  \centering
  \begin{tikzpicture}[scale=0.5]
    \fill[black] (0,2) rectangle (1,3);
    \fill[black] (1,2) rectangle (2,3);
    \fill[black] (2,2) rectangle (3,3);
    \fill[black] (0,1) rectangle (1,2);
    \fill[black] (1,1) rectangle (2,2);
    \draw[thick] (0,0) grid (3,3);
  \end{tikzpicture}
  \\[2pt] \small 1 (8)
\end{minipage}%
\hfill
\begin{minipage}[t]{0.22\textwidth}
  \centering
  \begin{tikzpicture}[scale=0.5]
    \fill[black] (0,2) rectangle (1,3);
    \fill[black] (1,2) rectangle (2,3);
    \fill[black] (2,2) rectangle (3,3);
    \fill[black] (0,1) rectangle (1,2);
    \fill[black] (2,1) rectangle (3,2);
    \draw[thick] (0,0) grid (3,3);
  \end{tikzpicture}
  \\[2pt] \small 2 (4)
\end{minipage}%
\hfill
\begin{minipage}[t]{0.22\textwidth}
  \centering
  \begin{tikzpicture}[scale=0.5]
    \fill[black] (0,2) rectangle (1,3);
    \fill[black] (1,2) rectangle (2,3);
    \fill[black] (2,2) rectangle (3,3);
    \fill[black] (0,1) rectangle (1,2);
    \fill[black] (0,0) rectangle (1,1);
    \draw[thick] (0,0) grid (3,3);
  \end{tikzpicture}
  \\[2pt] \small 3 (4)
\end{minipage}%
\hfill
\begin{minipage}[t]{0.22\textwidth}
  \centering
  \begin{tikzpicture}[scale=0.5]
    \fill[black] (0,2) rectangle (1,3);
    \fill[black] (1,2) rectangle (2,3);
    \fill[black] (2,2) rectangle (3,3);
    \fill[black] (0,1) rectangle (1,2);
    \fill[black] (1,0) rectangle (2,1);
    \draw[thick] (0,0) grid (3,3);
  \end{tikzpicture}
  \\[2pt] \small 4 (8)
\end{minipage}%

\vspace{12pt}

\noindent
\begin{minipage}[t]{0.22\textwidth}
  \centering
  \begin{tikzpicture}[scale=0.5]
    \fill[black] (0,2) rectangle (1,3);
    \fill[black] (1,2) rectangle (2,3);
    \fill[black] (2,2) rectangle (3,3);
    \fill[black] (0,1) rectangle (1,2);
    \fill[black] (2,0) rectangle (3,1);
    \draw[thick] (0,0) grid (3,3);
  \end{tikzpicture}
  \\[2pt] \small 5 (8)
\end{minipage}%
\hfill
\begin{minipage}[t]{0.22\textwidth}
  \centering
  \begin{tikzpicture}[scale=0.5]
    \fill[black] (0,2) rectangle (1,3);
    \fill[black] (1,2) rectangle (2,3);
    \fill[black] (2,2) rectangle (3,3);
    \fill[black] (1,1) rectangle (2,2);
    \fill[black] (0,0) rectangle (1,1);
    \draw[thick] (0,0) grid (3,3);
  \end{tikzpicture}
  \\[2pt] \small 6 (8)
\end{minipage}%
\hfill
\begin{minipage}[t]{0.22\textwidth}
  \centering
  \begin{tikzpicture}[scale=0.5]
    \fill[black] (0,2) rectangle (1,3);
    \fill[black] (1,2) rectangle (2,3);
    \fill[black] (2,2) rectangle (3,3);
    \fill[black] (1,1) rectangle (2,2);
    \fill[black] (1,0) rectangle (2,1);
    \draw[thick] (0,0) grid (3,3);
  \end{tikzpicture}
  \\[2pt] \small 7 (4)
\end{minipage}%
\hfill
\begin{minipage}[t]{0.22\textwidth}
  \centering
  \begin{tikzpicture}[scale=0.5]
    \fill[black] (0,2) rectangle (1,3);
    \fill[black] (1,2) rectangle (2,3);
    \fill[black] (2,2) rectangle (3,3);
    \fill[black] (0,0) rectangle (1,1);
    \fill[black] (1,0) rectangle (2,1);
    \draw[thick] (0,0) grid (3,3);
  \end{tikzpicture}
  \\[2pt] \small 8 (8)
\end{minipage}%

\vspace{12pt}

\noindent
\begin{minipage}[t]{0.22\textwidth}
  \centering
  \begin{tikzpicture}[scale=0.5]
    \fill[black] (0,2) rectangle (1,3);
    \fill[black] (1,2) rectangle (2,3);
    \fill[black] (2,2) rectangle (3,3);
    \fill[black] (0,0) rectangle (1,1);
    \fill[black] (2,0) rectangle (3,1);
    \draw[thick] (0,0) grid (3,3);
  \end{tikzpicture}
  \\[2pt] \small 9 (4)
\end{minipage}%
\hfill
\begin{minipage}[t]{0.22\textwidth}
  \centering
  \begin{tikzpicture}[scale=0.5]
    \fill[black] (0,2) rectangle (1,3);
    \fill[black] (1,2) rectangle (2,3);
    \fill[black] (0,1) rectangle (1,2);
    \fill[black] (1,1) rectangle (2,2);
    \fill[black] (2,1) rectangle (3,2);
    \draw[thick] (0,0) grid (3,3);
  \end{tikzpicture}
  \\[2pt] \small 10 (8)
\end{minipage}%
\hfill
\begin{minipage}[t]{0.22\textwidth}
  \centering
  \begin{tikzpicture}[scale=0.5]
    \fill[black] (0,2) rectangle (1,3);
    \fill[black] (1,2) rectangle (2,3);
    \fill[black] (0,1) rectangle (1,2);
    \fill[black] (1,1) rectangle (2,2);
    \fill[black] (2,0) rectangle (3,1);
    \draw[thick] (0,0) grid (3,3);
  \end{tikzpicture}
  \\[2pt] \small 11 (4)
\end{minipage}%
\hfill
\begin{minipage}[t]{0.22\textwidth}
  \centering
  \begin{tikzpicture}[scale=0.5]
    \fill[black] (0,2) rectangle (1,3);
    \fill[black] (1,2) rectangle (2,3);
    \fill[black] (0,1) rectangle (1,2);
    \fill[black] (2,1) rectangle (3,2);
    \fill[black] (1,0) rectangle (2,1);
    \draw[thick] (0,0) grid (3,3);
  \end{tikzpicture}
  \\[2pt] \small 12 (4)
\end{minipage}%

\vspace{12pt}

\noindent
\begin{minipage}[t]{0.22\textwidth}
  \centering
  \begin{tikzpicture}[scale=0.5]
    \fill[black] (0,2) rectangle (1,3);
    \fill[black] (1,2) rectangle (2,3);
    \fill[black] (0,1) rectangle (1,2);
    \fill[black] (2,1) rectangle (3,2);
    \fill[black] (2,0) rectangle (3,1);
    \draw[thick] (0,0) grid (3,3);
  \end{tikzpicture}
  \\[2pt] \small 13 (8)
\end{minipage}%
\hfill
\begin{minipage}[t]{0.22\textwidth}
  \centering
  \begin{tikzpicture}[scale=0.5]
    \fill[black] (0,2) rectangle (1,3);
    \fill[black] (1,2) rectangle (2,3);
    \fill[black] (1,1) rectangle (2,2);
    \fill[black] (2,1) rectangle (3,2);
    \fill[black] (0,0) rectangle (1,1);
    \draw[thick] (0,0) grid (3,3);
  \end{tikzpicture}
  \\[2pt] \small 14 (8)
\end{minipage}%
\hfill
\begin{minipage}[t]{0.22\textwidth}
  \centering
  \begin{tikzpicture}[scale=0.5]
    \fill[black] (0,2) rectangle (1,3);
    \fill[black] (1,2) rectangle (2,3);
    \fill[black] (1,1) rectangle (2,2);
    \fill[black] (2,1) rectangle (3,2);
    \fill[black] (1,0) rectangle (2,1);
    \draw[thick] (0,0) grid (3,3);
  \end{tikzpicture}
  \\[2pt] \small 15 (8)
\end{minipage}%
\hfill
\begin{minipage}[t]{0.22\textwidth}
  \centering
  \begin{tikzpicture}[scale=0.5]
    \fill[black] (0,2) rectangle (1,3);
    \fill[black] (1,2) rectangle (2,3);
    \fill[black] (1,1) rectangle (2,2);
    \fill[black] (2,1) rectangle (3,2);
    \fill[black] (2,0) rectangle (3,1);
    \draw[thick] (0,0) grid (3,3);
  \end{tikzpicture}
  \\[2pt] \small 16 (4)
\end{minipage}%

\vspace{12pt}

\noindent
\begin{minipage}[t]{0.22\textwidth}
  \centering
  \begin{tikzpicture}[scale=0.5]
    \fill[black] (0,2) rectangle (1,3);
    \fill[black] (1,2) rectangle (2,3);
    \fill[black] (1,1) rectangle (2,2);
    \fill[black] (0,0) rectangle (1,1);
    \fill[black] (1,0) rectangle (2,1);
    \draw[thick] (0,0) grid (3,3);
  \end{tikzpicture}
  \\[2pt] \small 17 (4)
\end{minipage}%
\hfill
\begin{minipage}[t]{0.22\textwidth}
  \centering
  \begin{tikzpicture}[scale=0.5]
    \fill[black] (0,2) rectangle (1,3);
    \fill[black] (1,2) rectangle (2,3);
    \fill[black] (1,1) rectangle (2,2);
    \fill[black] (0,0) rectangle (1,1);
    \fill[black] (2,0) rectangle (3,1);
    \draw[thick] (0,0) grid (3,3);
  \end{tikzpicture}
  \\[2pt] \small 18 (8)
\end{minipage}%
\hfill
\begin{minipage}[t]{0.22\textwidth}
  \centering
  \begin{tikzpicture}[scale=0.5]
    \fill[black] (0,2) rectangle (1,3);
    \fill[black] (1,2) rectangle (2,3);
    \fill[black] (1,1) rectangle (2,2);
    \fill[black] (1,0) rectangle (2,1);
    \fill[black] (2,0) rectangle (3,1);
    \draw[thick] (0,0) grid (3,3);
  \end{tikzpicture}
  \\[2pt] \small 19 (4)
\end{minipage}%
\hfill
\begin{minipage}[t]{0.22\textwidth}
  \centering
  \begin{tikzpicture}[scale=0.5]
    \fill[black] (0,2) rectangle (1,3);
    \fill[black] (1,2) rectangle (2,3);
    \fill[black] (2,1) rectangle (3,2);
    \fill[black] (0,0) rectangle (1,1);
    \fill[black] (1,0) rectangle (2,1);
    \draw[thick] (0,0) grid (3,3);
  \end{tikzpicture}
  \\[2pt] \small 20 (4)
\end{minipage}%

\vspace{12pt}

\noindent
\begin{minipage}[t]{0.22\textwidth}
  \centering
  \begin{tikzpicture}[scale=0.5]
    \fill[black] (0,2) rectangle (1,3);
    \fill[black] (1,2) rectangle (2,3);
    \fill[black] (2,1) rectangle (3,2);
    \fill[black] (0,0) rectangle (1,1);
    \fill[black] (2,0) rectangle (3,1);
    \draw[thick] (0,0) grid (3,3);
  \end{tikzpicture}
  \\[2pt] \small 21 (4)
\end{minipage}%
\hfill
\begin{minipage}[t]{0.22\textwidth}
  \centering
  \begin{tikzpicture}[scale=0.5]
    \fill[black] (0,2) rectangle (1,3);
    \fill[black] (2,2) rectangle (3,3);
    \fill[black] (1,1) rectangle (2,2);
    \fill[black] (0,0) rectangle (1,1);
    \fill[black] (2,0) rectangle (3,1);
    \draw[thick] (0,0) grid (3,3);
  \end{tikzpicture}
  \\[2pt] \small 22 (1)
\end{minipage}%
\hfill
\begin{minipage}[t]{0.22\textwidth}
  \centering
  \begin{tikzpicture}[scale=0.5]
    \fill[black] (1,2) rectangle (2,3);
    \fill[black] (0,1) rectangle (1,2);
    \fill[black] (1,1) rectangle (2,2);
    \fill[black] (2,1) rectangle (3,2);
    \fill[black] (1,0) rectangle (2,1);
    \draw[thick] (0,0) grid (3,3);
  \end{tikzpicture}
  \\[2pt] \small 23 (1)
\end{minipage}%
\caption{The orbits of 5 elements of a $3 \times 3$ grid under the action of the dihedral group $D_4$, corresponding to orbits of coordinate projections $\mathbb{P}^8 \dashrightarrow \mathbb{P}^4$ under the automorphism group.  Labeled with orbit number (orbit length).}
\label{fig:5}
\end{figure}
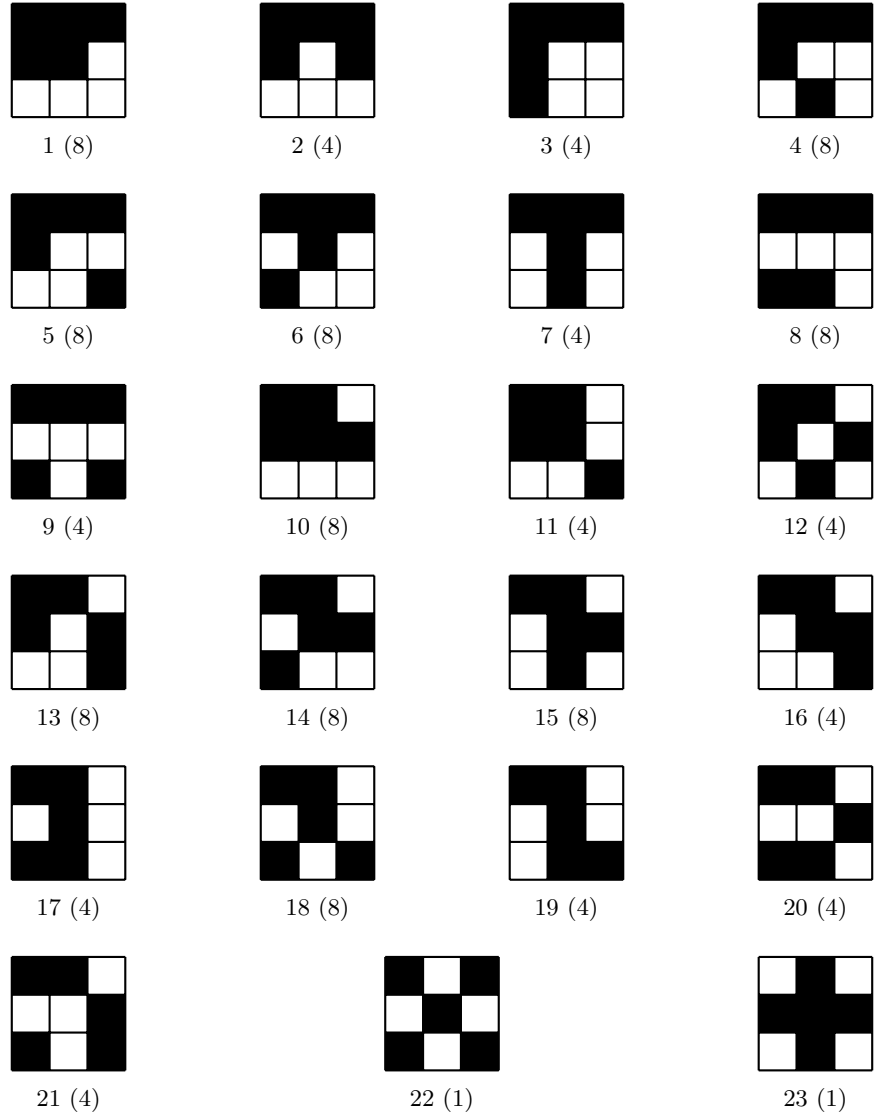

\begin{prop}
The image of each coordinate projection $p\colon V \to \mathbb{P}^4$ is an intersection of two quadrics, whose associated pencil has discriminant supported on $\Q$-rational points.  Of the 23 orbits of possible projections, there are 5 geometric isomorphism classes of magic quartic del Pezzo surface (summarized in Table~\ref{tab:del-pezzo}).
\end{prop}

\begin{proof}
    Each of the images of the projections $p\colon V \to \mathbb{P}^4$ described above has 2 equations given by diagonal quadratic forms.
    We first recover the equations uniformly across all $23$ orbits.  Writing $X_v=v^{2}$, the eight defining relations of $V$ are linear in the $X_v$, so we may collect them in a coefficient matrix $\Lambda\in\operatorname{Mat}_{8\times9}(\mathbb{Q})$ with $\Lambda\,(X_A,\dots,X_H)^{\mathsf T}=0$. So, $V=\nu^{-1}\bigl(\{\Lambda X=0\}\bigr)$ for the squaring map $\nu\colon[A:\cdots:H]\mapsto[A^{2}:\cdots:H^{2}]$.  Choosing an orbit amounts to choosing a five–element subset $T$ of the coordinates, with complement $T^{c}$. Furthermore, the relations among the chosen squares that survive the projection are the combinations $c\Lambda$ of the defining relations that kill the four forgotten columns, i.e., those with $c\,\Lambda_{T^{c}}=0$.  Evaluating such a $c\Lambda$ on the remaining columns yields a diagonal quadric in the variables $T$ vanishing on the image; a direct computation \cite{singer_2026} shows that the space of these is two–dimensional in every orbit. Since the center of each projection is disjoint from $V$ and the resulting intersection is irreducible, the morphism $p$ has image the complete intersection $Q_{1}\cap Q_{2}$ of the two resulting quadrics. 
    
    To see this in action, we consider orbit 1 (in Figure~\ref{fig:5}). Indeed, we can rewrite the square as
    \[
\magicsquarethree{A^2}{B^2}{C^2}{D^2}{M^2}{2M^2 - D^2}{2M^2 - C^2}{2M^2 - B^2}{2M^2 - A^2}
\]
This gives that the image of the projection
\begin{align*}
    p \colon V &{}\longrightarrow S_1 \subset \mathbb{P}^4 \\
    [A:B:C:D:M:E:F:G:H] &{} \longmapsto [A:B:C:D:M]
\end{align*}
has equations
\begin{align} 
\label{eq:MagicdelPezzoEquations}
\begin{split}
A^2 + B^2 + C^2 &{} = 3M^2 \\
A^2 + D^2 - C^2 &{} = M^2.
\end{split}
\end{align}
    Our method above returns the pencil spanned by $A^{2}+D^{2}-C^{2}-M^{2}$ and $B^{2}+2C^{2}-D^{2}-2M^{2}$, recovering \eqref{eq:MagicdelPezzoEquations}. The process for determining the equations of the other magic del Pezzo surfaces corresponding to the remaining orbits is analogous. 

Since the two quadrics are already diagonal over $\Q$ in each of the 23 orbits, we deduce that the associated pencil has discriminant supported in $\Q$-points, which can be directly computed from the diagonal coefficients of the two quadratic forms.  Indeed, the five points with multiplicity, considered as points in $\P^1$, are $(c_1^i : c_2^i)$, where $c_j^i$ is the coefficient in front of $x_i^2$ in the quadratic form $q_j$, where $(x_0:\dotsm:x_4)$ are the homogeneous coordinates on $\P^4$. We use Magma to compute Table~\ref{tab:del-pezzo}, with code provided to verify our claims \cite{singer_2026}.
\end{proof}

\begin{table}[!htbp]
\centering
\renewcommand{\arraystretch}{1.5}
\setlength{\tabcolsep}{8pt}
\begin{adjustbox}{max width=\textwidth}
\begin{tabular}{@{}l l c c c l@{}}
\toprule
Orbits & Equations & Segre Symbol & \# Nodes & \# Lines & Discriminant Config \\
\toprule
$\begin{aligned}[t]
1, 2, 3 \\ 8, 9
\end{aligned}$
  & $\begin{aligned}[t]
       x_1^2 + 2x_2^2 &= x_3^2 + 2x_4^2 \\
       x_0^2 + x_3^2  &= x_2^2 + x_4^2
     \end{aligned}$
  & $[1,1,1,1,1]$ & $0$ & $16$
  & $\{-2,\, -1,\, 0,\, 2, \, \infty\}$ \\\hline
$4, 6, 7$
  & $\begin{aligned}[t]
       2x_2^2          &= x_3^2 + x_4^2 \\
       2x_0^2 + x_3^2  &= x_1^2 + 2x_4^2
     \end{aligned}$
  & $[(1,1),1,1,1]$ & $2$ & $8$
  & $\{-1,\, 0, \, 0,\, \frac{1}{2}, \, \infty\}$ \\\hline
$5$
  & $\begin{aligned}[t]
       x_1^2 + x_3^2   &= 2x_4^2 \\
       x_0^2 + 2x_3^2  &= x_4^2 + 2x_2^2
     \end{aligned}$
  & $[(1,1),1,1,1]$ & $2$ & $8$
  & $\{  0,\, 0, \, \frac{1}{2},\, 2, \, \infty\}$ \\\hline
$\begin{aligned}[t]
10, 12, 13 \\
14, 16, 17 \\
18, 21, 22
\end{aligned}$
  & $\begin{aligned}[t]
       x_2^2 + x_4^2   &= 2x_3^2 \\
       2x_0^2 + x_1^2  &= 2x_3^2 + x_4^2
     \end{aligned}$
  & $[(1,1),1,1,1]$ & $2$ & $8$
  & $\{-1,\, 0,\, 0, \, 1,\, \infty\}$ \\\hline
$\begin{aligned}[t]
11, 15, 19, \\ 20, 23
\end{aligned}$
  & $\begin{aligned}[t]
       x_1^2 + x_2^2  &= 2x_4^2 \\
       x_0^2 + x_4^2  &= 2x_3^2
     \end{aligned}$
  & $[(1,1),(1,1),1]$ & $4$ & $4$
  & $\{-2,\,  0,\, 0,\, \infty, \, \infty\}$ \\
\bottomrule
\end{tabular}
\end{adjustbox}
\caption{Geometric isomorphism classes of magic quartic del Pezzo surfaces organized by coordinate projection orbit (see Figure~\ref{fig:5}), sample equations in $\P^4$ for a single representative in some orbit, Segre symbol, number of nodes, number of lines (over $\Qbar$), and discriminant configuration of the associated pencil.}
\label{tab:del-pezzo}
\end{table}

We can consider rational curves on their magic del Pezzo surfaces, and their preimages under the projection maps $V \to \mathbb{P}^4$ to produce more divisors on $\mathbb{P}^4$. Most fruitful of these are the lines.

\begin{prop}
The 1148 lines over $\Qbar$ on the 126 magic quartic del Pezzo surfaces yield divisors on $V$ over $\Qbar$.
\end{prop}
\begin{proof}
This uses the same method as the proof of Proposition \ref{prop:EmptyLines}, but with $\mathbb{G}(1, 4)$ instead of $\mathbb{G}(1, 8)$. Code to compute these can be found at \cite{singer_2026}. In terms of enumeration, summing the lines defined over $\mathbb{Q}$ across the $126$ coordinate projections, grouped by the five geometric isomorphism classes
of Table~\ref{tab:del-pezzo}, gives
\[
  28 \cdot 16 \;+\; (20 + 8 + 49)\cdot 8 \;+\; 21 \cdot 4
  \;=\; 448 + 616 + 84
  \;=\; 1148
\]
lines over $\mathbb{Q}$, where the constants $28, 20, 8, 49, 21$ are the
numbers of coordinate projections falling into each geometric class.
\end{proof}

Most of these divisors do not contribute new classes. The preimages of lines from orbits $4$, $6$, $7$, and $10$--$23$ coincide, as divisor classes on $\widetilde{V}$, with irreducible components of the split hyperplane sections already constructed in Section~\ref{sec:divisors}; we verify these coincidences in Magma \cite{singer_2026}. The new classes arise only from orbits $1$, $2$, $3$, $5$, $8$, and $9$, whose line-preimages
account for
\[
  128 + 64 + 64 + 64 + 128 + 64 = 512
\]
divisors. Of these, $128$ are redundant, either coinciding with one
another under the action of $\operatorname{Aut}(V)$ or with previously
constructed divisors. This leaves $384$ distinct new divisor classes defined over $\mathbb{Q}(i, \sqrt{2}, \sqrt{3}, \sqrt{5})$. We emphasize that these $384$ count divisor classes, not line-preimages. Their irreducible components are smooth curves of genus $17$ and degree $16$, or genus $5$ and degree $8$. For example, the hyperplane sections $\sqrt{2}A \pm \sqrt{2}B \pm \sqrt{3}D = 0$ are the preimages of the lines on the surface of orbit~$3$. These preimages are essential to the proof of Theorem~\ref{thm:picard-rank}.

We can relate the geometry of the 128 rational conics that are the irreducible components of the 3 distinct entry locus $Z_3 \subset V$, see Section~\ref{subsec:nondistinct}, to conics on the magic quartic del Pezzo surfaces. 

\begin{prop}
For each magic quartic del Pezzo surface $p : V \to S$, the image of the union of the 128 conics in $Z_3$ is the union of eight conics in $S$. 
\end{prop}

\begin{proof}
We recall that there are two families of these conics. The first is defined by a number of hyperplane sections and the equation
\[
   B^2 + C^2 = 2A^2
\]
and the second is defined by a number of hyperplane sections and the equation
\[
    A^2 + B^2 = 2M^2.
\]
Intersecting the equations of the del Pezzo surface $S_1$ corresponding to orbit 1, see Equation \ref{eq:MagicdelPezzoEquations}, with the hyperplanes $A = \pm M$ reduces the equations of $S_1$ to
\begin{align*}
    C^2 = &{\,} D^2 \\
    B^2 + C^2 = &{\,}2A^2,  
\end{align*}
which defines four rational conics $C_1,\dots, C_4 \subset S_1$. The method to check this for the other orbits is similar.
\end{proof}

\begin{remark}
However, the fibers of rational points of magic quartic del Pezzo surfaces are not entirely contained in $Z_3$. For example, the rational points of height 28 on $S_1$ all lift to points on $V$ corresponding to the magic square with 5 squares and distinct entries
\[
\magicsquarethree{29^2}{35^2}{11^2}{3^2}{27^2}{1449}{1337}{233}{617}
\]
with magic sum 2187.
\end{remark}

We can project from a point on each magic del Pezzo surface to obtain magic cubic surfaces, yielding more divisors on $V$ via their preimages. Again, the most fruitful of these are the lines.

\begin{prop}
    There are 126 magic cubic surfaces yielded as projections from magic quartic del Pezzo surfaces. The 1981 lines defined over $\overline{\mathbb{Q}}$ on them yield divisors on $V$ over $\overline{\mathbb{Q}}$.
\end{prop}

\begin{proof}
    Projecting from the smooth rational point $[1,1,1,1,1]$ on any magic quartic del Pezzo surface yields a magic cubic surface. The number of lines on each cubic surface is determined by the number of nodes, which we group by orbit in Figure~\ref{fig:5}:
    \begin{center}
        \begin{tabular}{c|c|c}
        Orbit & $\#$Nodes on Associated Cubic& $\#$Lines \\  
        \hline 
        1, 2, 3, 8, 9 & 0 & 27 \\
        4, 5, 6, 7 & 2 & 16 \\
        10, 12, 13, 14, 16, 17, 18, 21, 22& 3 & 12 \\
        11, 15, 19, 20, 23 & 4 & 9
        \end{tabular}
    \end{center}
    Accounting for the length of each orbit under the automorphism group, this amounts to
    \[
    (28\cdot27) + (28\cdot16) + (49\cdot12) + (21\cdot 9) = 756 + 448 + 588 + 189 =  1981
    \]
    lines in total. Again, we use the same method as the proof of Proposition \ref{prop:EmptyLines} to compute equations for the lines on each magic cubic surface, but with $\mathbb{G}(1, 3)$ instead of $\mathbb{G}(1, 8)$. Code to compute these, and their preimages on $V$, can be found at \cite{singer_2026}.
\end{proof}

Although the above yields 1981 divisors on $V$ by taking preimages, many are either equal to each other, reducible, or already accounted for by the line preimages on magic quartic del Pezzo surfaces or the hyperplane sections constructed in Section \ref{sec:divisors}. In total, we only obtain 148 new irreducible divisors as preimages of lines on magic cubic surfaces.

\begin{remark}
We note that the images of coordinate projections $V \to \mathbb{P}^3$ are magic quadric surfaces. These do not produce any new divisors to help with computing the Picard lattice of $V$.
\end{remark}

\subsection{Magic K3 Surfaces}
\label{subsec:magic_K3}

Analogously to magic squares with 5 square entries, for the case of magic squares with 6 square entries, we will compute defining equations for the image of $V$ under coordinate projections $\mathbb{P}^8 \dashrightarrow \mathbb{P}^5$.  As before, we can consider the orbits of the 84 coordinate projections under the action of the automorphism group, which is in bijection with the orbits of 6 entries of a $3 \times 3$ grid under the natural action of the dihedral group. By Burnside's lemma, there are 16 orbits, with representatives given in Figure~\ref{fig:6}, cf.\ \cite[Section~1]{bremnersquaresii}.

\begin{figure}
\vspace{2mm}
\noindent
\begin{minipage}[t]{0.22\textwidth}
  \centering
  \begin{tikzpicture}[scale=0.5]
    \fill[black] (0,2) rectangle (1,3);
    \fill[black] (1,2) rectangle (2,3);
    \fill[black] (2,2) rectangle (3,3);
    \fill[black] (0,1) rectangle (1,2);
    \fill[black] (1,1) rectangle (2,2);
    \fill[black] (0,0) rectangle (1,1);
    \draw[thick] (0,0) grid (3,3);
  \end{tikzpicture}
  \\[2pt] \small 1 (4)
\end{minipage}%
\hfill
\begin{minipage}[t]{0.22\textwidth}
  \centering
  \begin{tikzpicture}[scale=0.5]
    \fill[black] (1,2) rectangle (2,3);
    \fill[black] (2,2) rectangle (3,3);
    \fill[black] (0,1) rectangle (1,2);
    \fill[black] (2,1) rectangle (3,2);
    \fill[black] (0,0) rectangle (1,1);
    \fill[black] (1,0) rectangle (2,1);
    \draw[thick] (0,0) grid (3,3);
  \end{tikzpicture}
  \\[2pt] \small 2 (2)
\end{minipage}%
\hfill
\begin{minipage}[t]{0.22\textwidth}
  \centering
  \begin{tikzpicture}[scale=0.5]
    \fill[black] (0,2) rectangle (1,3);
    \fill[black] (1,2) rectangle (2,3);
    \fill[black] (2,2) rectangle (3,3);
    \fill[black] (0,0) rectangle (1,1);
    \fill[black] (1,0) rectangle (2,1);
    \fill[black] (2,0) rectangle (3,1);
    \draw[thick] (0,0) grid (3,3);
  \end{tikzpicture}
  \\[2pt] \small 3 (2)
\end{minipage}%
\hfill
\begin{minipage}[t]{0.22\textwidth}
  \centering
  \begin{tikzpicture}[scale=0.5]
    \fill[black] (1,2) rectangle (2,3);
    \fill[black] (0,1) rectangle (1,2);
    \fill[black] (1,1) rectangle (2,2);
    \fill[black] (2,1) rectangle (3,2);
    \fill[black] (0,0) rectangle (1,1);
    \fill[black] (2,0) rectangle (3,1);
    \draw[thick] (0,0) grid (3,3);
  \end{tikzpicture}
  \\[2pt] \small 4 (4)
\end{minipage}%

\vspace{12pt}

\noindent
\begin{minipage}[t]{0.22\textwidth}
  \centering
  \begin{tikzpicture}[scale=0.5]
    \fill[black] (0,1) rectangle (1,2);
    \fill[black] (1,1) rectangle (2,2);
    \fill[black] (2,1) rectangle (3,2);
    \fill[black] (0,0) rectangle (1,1);
    \fill[black] (1,0) rectangle (2,1);
    \fill[black] (2,0) rectangle (3,1);
    \draw[thick] (0,0) grid (3,3);
  \end{tikzpicture}
  \\[2pt] \small 5 (4)
\end{minipage}%
\hfill
\begin{minipage}[t]{0.22\textwidth}
  \centering
  \begin{tikzpicture}[scale=0.5]
    \fill[black] (1,2) rectangle (2,3);
    \fill[black] (0,1) rectangle (1,2);
    \fill[black] (1,1) rectangle (2,2);
    \fill[black] (2,1) rectangle (3,2);
    \fill[black] (0,0) rectangle (1,1);
    \fill[black] (1,0) rectangle (2,1);
    \draw[thick] (0,0) grid (3,3);
  \end{tikzpicture}
  \\[2pt] \small 6 (4)
\end{minipage}%
\hfill
\begin{minipage}[t]{0.22\textwidth}
  \centering
  \begin{tikzpicture}[scale=0.5]
    \fill[black] (2,2) rectangle (3,3);
    \fill[black] (0,1) rectangle (1,2);
    \fill[black] (1,1) rectangle (2,2);
    \fill[black] (2,1) rectangle (3,2);
    \fill[black] (0,0) rectangle (1,1);
    \fill[black] (2,0) rectangle (3,1);
    \draw[thick] (0,0) grid (3,3);
  \end{tikzpicture}
  \\[2pt] \small 7 (8)
\end{minipage}%
\hfill
\begin{minipage}[t]{0.22\textwidth}
  \centering
  \begin{tikzpicture}[scale=0.5]
    \fill[black] (1,2) rectangle (2,3);
    \fill[black] (2,2) rectangle (3,3);
    \fill[black] (1,1) rectangle (2,2);
    \fill[black] (2,1) rectangle (3,2);
    \fill[black] (0,0) rectangle (1,1);
    \fill[black] (2,0) rectangle (3,1);
    \draw[thick] (0,0) grid (3,3);
  \end{tikzpicture}
  \\[2pt] \small 8 (8)
\end{minipage}%

\vspace{12pt}

\noindent
\begin{minipage}[t]{0.22\textwidth}
  \centering
  \begin{tikzpicture}[scale=0.5]
    \fill[black] (2,2) rectangle (3,3);
    \fill[black] (0,1) rectangle (1,2);
    \fill[black] (2,1) rectangle (3,2);
    \fill[black] (0,0) rectangle (1,1);
    \fill[black] (1,0) rectangle (2,1);
    \fill[black] (2,0) rectangle (3,1);
    \draw[thick] (0,0) grid (3,3);
  \end{tikzpicture}
  \\[2pt] \small 9 (8)
\end{minipage}%
\hfill
\begin{minipage}[t]{0.22\textwidth}
  \centering
  \begin{tikzpicture}[scale=0.5]
    \fill[black] (2,2) rectangle (3,3);
    \fill[black] (0,1) rectangle (1,2);
    \fill[black] (1,1) rectangle (2,2);
    \fill[black] (2,1) rectangle (3,2);
    \fill[black] (1,0) rectangle (2,1);
    \fill[black] (2,0) rectangle (3,1);
    \draw[thick] (0,0) grid (3,3);
  \end{tikzpicture}
  \\[2pt] \small 10 (8)
\end{minipage}%
\hfill
\begin{minipage}[t]{0.22\textwidth}
  \centering
  \begin{tikzpicture}[scale=0.5]
    \fill[black] (2,2) rectangle (3,3);
    \fill[black] (0,1) rectangle (1,2);
    \fill[black] (1,1) rectangle (2,2);
    \fill[black] (2,1) rectangle (3,2);
    \fill[black] (0,0) rectangle (1,1);
    \fill[black] (1,0) rectangle (2,1);
    \draw[thick] (0,0) grid (3,3);
  \end{tikzpicture}
  \\[2pt] \small 11 (8)
\end{minipage}%
\hfill
\begin{minipage}[t]{0.22\textwidth}
  \centering
  \begin{tikzpicture}[scale=0.5]
    \fill[black] (1,2) rectangle (2,3);
    \fill[black] (0,1) rectangle (1,2);
    \fill[black] (2,1) rectangle (3,2);
    \fill[black] (0,0) rectangle (1,1);
    \fill[black] (1,0) rectangle (2,1);
    \fill[black] (2,0) rectangle (3,1);
    \draw[thick] (0,0) grid (3,3);
  \end{tikzpicture}
  \\[2pt] \small 12 (4)
\end{minipage}%

\vspace{12pt}

\noindent
\begin{minipage}[t]{0.22\textwidth}
  \centering
  \begin{tikzpicture}[scale=0.5]
    \fill[black] (1,2) rectangle (2,3);
    \fill[black] (2,2) rectangle (3,3);
    \fill[black] (0,1) rectangle (1,2);
    \fill[black] (1,1) rectangle (2,2);
    \fill[black] (0,0) rectangle (1,1);
    \fill[black] (2,0) rectangle (3,1);
    \draw[thick] (0,0) grid (3,3);
  \end{tikzpicture}
  \\[2pt] \small 13 (4)
\end{minipage}%
\hfill
\begin{minipage}[t]{0.22\textwidth}
  \centering
  \begin{tikzpicture}[scale=0.5]
    \fill[black] (0,2) rectangle (1,3);
    \fill[black] (2,2) rectangle (3,3);
    \fill[black] (1,1) rectangle (2,2);
    \fill[black] (2,1) rectangle (3,2);
    \fill[black] (0,0) rectangle (1,1);
    \fill[black] (2,0) rectangle (3,1);
    \draw[thick] (0,0) grid (3,3);
  \end{tikzpicture}
  \\[2pt] \small 14 (4)
\end{minipage}%
\hfill
\begin{minipage}[t]{0.22\textwidth}
  \centering
  \begin{tikzpicture}[scale=0.5]
    \fill[black] (1,2) rectangle (2,3);
    \fill[black] (2,2) rectangle (3,3);
    \fill[black] (0,1) rectangle (1,2);
    \fill[black] (0,0) rectangle (1,1);
    \fill[black] (1,0) rectangle (2,1);
    \fill[black] (2,0) rectangle (3,1);
    \draw[thick] (0,0) grid (3,3);
  \end{tikzpicture}
  \\[2pt] \small 15 (8)
\end{minipage}%
\hfill
\begin{minipage}[t]{0.22\textwidth}
  \centering
  \begin{tikzpicture}[scale=0.5]
    \fill[black] (0,2) rectangle (1,3);
    \fill[black] (2,2) rectangle (3,3);
    \fill[black] (2,1) rectangle (3,2);
    \fill[black] (0,0) rectangle (1,1);
    \fill[black] (1,0) rectangle (2,1);
    \fill[black] (2,0) rectangle (3,1);
    \draw[thick] (0,0) grid (3,3);
  \end{tikzpicture}
  \\[2pt] \small 16 (4)
\end{minipage}%
\caption{The orbits of 6 elements of a $3 \times 3$ grid under the action of the
dihedral group $D_4$, corresponding to orbits of coordinate projections
$\mathbb{P}^8 \dashrightarrow \mathbb{P}^5$ under the action of the dihedral group.  Each
orbit corresponds to an isomorphism class of magic octic K3 surface.  The ordering follows the sixteen orbits I--XVI of Bremner \cite[p.~290]{bremnersquaresii}, though with the colors reversed.  Labeled with orbit number (orbit length).}
\label{fig:6}
\end{figure}
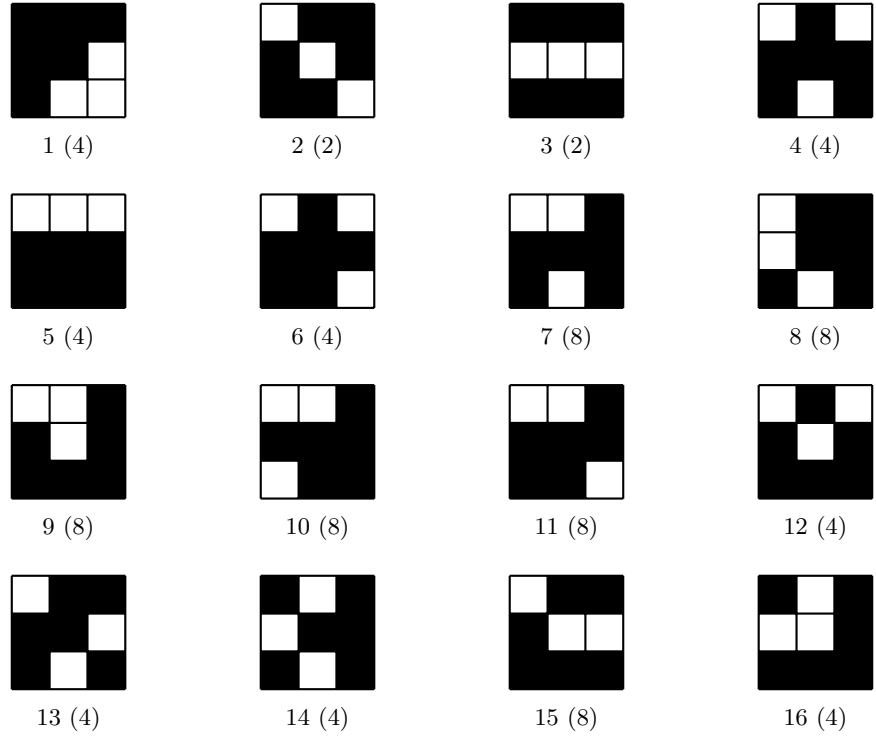

Each orbit in Figure~\ref{fig:6} determines a coordinate projection that restricts to a morphism $V \to \P^5$, whose image turns out to be a complete intersection of three quadrics in $\P^5$, the projective model of a K3 surface of degree 8, which we call a \linedef{magic octic K3 surface}.  Equations for, and data about the singularities of, the magic octic K3 surfaces is provided in Table~\ref{tab:k3-merged}.  The magic octic K3 surfaces were first studied by Bremner~\cite{bremnersquaresii}, who made a detailed study of the one corresponding to orbit 3 (the only smooth one); below we refer to this as \linedef{Bremner's octic K3 surface} and denote it by~$Y$.  Bremner~\cite[Section~3]{bremnersquaresii} proved that $Y_{\Qbar}$ has 8 conics, 32 lines, and maximal Picard rank 20.  He also gave~\cite[pp.~297--298]{bremnersquaresii} an elliptic fibration on the K3 surface corresponding to orbit 7, denoted below by $X$. In the following, we compute the geometric Picard rank of $X$. We also consider degree 2 K3 surfaces naturally associated to the magic octic K3 surfaces, showing that they are double covers of $\mathbb{P}^2$ ramified over special arrangements of six lines. Finally, using products of coordinate projections, we construct degree 12 magic K3 surfaces as well as magic Enriques surfaces whose K3 double cover is~$Y$.

As an example, we consider the projection 
\begin{align*}
p\colon V &{} \longrightarrow  \mathbb{P}^5 \\
[A:B:C:D:M:E:F:G:H] &{} \longmapsto [A:B:C:M:F:G].
\end{align*}
corresponding to orbit 7 (see Table~\ref{tab:k3-merged}) with image $X \subset \P^5$.  Letting the coordinates on $\mathbb{P}^5$ be $[x_0, x_1, x_2, x_3, x_4, x_5]$, we compute that $X$ has equations 
\begin{align*}
x_0^2 + x_3^2 - x_4^2 - x_5^2  &= 0 \\
x_1^2 - 2x_3^2 + x_5^2         &= 0 \\
x_2^2 - 2x_3^2 + x_4^2         &= 0 
\end{align*}
Using a similar method of point searching as before, we can generate many squares of 6 squares by taking the fibers of the finite morphism $V \to X$  over rational points with distinct entries, such as the square
\[
\magicsquarethree{653^2}{85^2}{329^2}{-137543}{425^2}{498793}{503^2}{595^2}{-65159}
\]
with magic sum 541875. This approach explains why the magic square with 7 squares in \eqref{eq:7squares} is so special. Indeed, it appears in the fiber of the rational point 
$[373 : 289 : 565 : 425 : 205 : 527]$ of $X$, and it's purely a coincidence that $E$ is a square in this fiber. Previous work of Bremner~\cite{bremnersquaresii} also shows that squares of 7 squares arising from this orbit are rational points on high genus hyperelliptic curves on $X$, giving an alternate explanation.

Beyond its use in producing interesting magic squares, $X$ has intriguing arithmetic itself.

\begin{prop}
\label{prop:K3_Pic_rank}
The geometric Picard rank $\rho(\widetilde{X}_{\overline{\mathbb{Q}}}) = 19$, where $\widetilde{X}_{\overline{\mathbb{Q}}}$ is the resolution of $X_{\overline{\mathbb{Q}}}$.
\end{prop}

\begin{proof}
    The model $X \subset \P^5$ has eight ordinary double points all defined over $\mathbb{Q}(i)$, all of which are images of the singular points on $V$ under the projection $p$. Taking the exceptional divisors from these singular points, along with 48 curves that are components of the reducible hyperplane sections defined by
\begin{align*}
    x_0 \pm x_3 = &{\,}0,  \quad &{} x_1 \pm \sqrt{-2}x_3 = &{\,} 0, \\
    x_0 \pm x_4 = &{\,}0, \quad  &{} x_3 \pm ix_4 = &{\,}0, \\
    x_1 \pm x_2 =&{\,} 0, \quad  &{} x_3 \pm ix_5= &{\,} 0, \\
    x_2 \pm x_3 = &{\,}0, &&
\end{align*}
gives an intersection matrix of rank 19. This yields $\rho(\widetilde{X}_{\overline{\mathbb{Q}}}) \geq 19$. Magma code to check this is provided \cite{singer_2026}.

To conclude, we give an upper bound on the Picard rank by working over finite fields, using the approach of Van Luijk~\cite{van2007k3} as refined by Kloosterman~\cite{kloosterman:elliptic} (see \cite[Section~2]{varilly2017arithmetic}). We begin by finding a degree 2 quasipolarization of $X$, which we can use to efficiently compute the Weil polynomial via Magma's built-in routine. In a geometric construction similar to that in \cite{auel2026noether}, projecting from the plane containing the conic $C \subset X$ given by the equations
\[
x_1^2 - 2x_4^2 + x_5^2 = x_0 + x_5 = x_2 + x_4 = x_3 + x_4 = 0
\]
gives a degree 2 cover
\begin{align*}
\pi\colon X &{} \longrightarrow \mathbb{P}^2,  \\
[x_0, x_1, x_2, x_3, x_4, x_5] &{} \longmapsto [x_0 + x_5, x_2+x_4, x_3+x_4].
\end{align*}
Letting the coordinates on $\mathbb{P}^2$ be $[x, y, z]$, we compute that $X$ is ramified over the sextic curve
\begin{align*}
x^4y^2 - 2x^2y^4 - 4x^4yz + 6x^2y^3z + 4x^4z^2 - 10x^2y^2z^2 &{}  \\
+ \ y^4z^2 + 12x^2yz^3 - 2y^3z^3 - 8x^2z^4 + y^2z^4= &{\, }0.
\end{align*}
This branch sextic splits into a union of Galois-conjugate smooth cubics 
\begin{align*}
    E_1 = &{\,} x^2y - \sqrt2xy^2 - 2x^2z + 2\sqrt2xyz - y^2z - 2\sqrt2xz^2 + yz^2,\\
    E_2 = &{\,} x^2y + \sqrt2xy^2 - 2x^2z - 2\sqrt2xyz - y^2z + 2\sqrt2xz^2 + yz^2.
\end{align*}
We remark that these cubics are quadratic twists of the elliptic curve
\begin{center}
    $y^2 = x^3 - x^2 - 9x + 9$.
\end{center}

Modulo 13, a prime of good reduction of $X$, we compute that the characteristic polynomial of Frobenius acting on $H^2_\et(\widetilde{X}_{\overline{\mathbb{F}}_{13}}, \mathbb{Q}_\ell)$ is $(x-13)^{20}(x^2 + 22x + 169)$.
Modulo 37, also a prime of good reduction of $X$, we compute that the characteristic polynomial of Frobenius is $(x-37)^{20}(x^2 + 38x + 1369)$.
The Weil polynomials and Frobenius traces of $X$ at all primes up to 37 can be found in full at \cite{singer_2026}.

The Tate conjecture for K3 surfaces over a finite field, together with Magma's Artin--Tate formula procedure, yield that
\[
\rho(\widetilde{X}_{\overline{\mathbb{F}}_{13}}) = 20, \quad \mathrm{disc}(\mathrm{Pic}(\widetilde{X}_{\overline{\mathbb{F}}_{13}})) = 3 \in \Q^\times/\Q^{\times 2}
\]
and
\[
\rho(\widetilde{X}_{\overline{\mathbb{F}}_{37}}) = 20, \quad  \mathrm{disc}(\mathrm{Pic}(\widetilde{X}_{\overline{\mathbb{F}}_{37}})) = 7 \in \Q^\times/\Q^{\times 2}.
\]
The discriminants of the geometric N\'eron--Severi lattice of $\widetilde{X}$ modulo 13 and 37 differ in $\mathbb{Q}^{\times}/\mathbb{Q}^{\times2}$, the smallest such primes yielding differing square classes. This suffices to give an upper bound of 19 on $\rho(\widetilde{X}_{\overline{\mathbb{Q}}})$ by \cite[Proposition~2.17]{varilly2017arithmetic}.
\end{proof}

\begin{remark}
Using the same methods as Proposition~\ref{prop:K3_Pic_rank}, we compute the geometric Picard rank of each of the magic octic K3 surfaces, and they are listed in Table~\ref{tab:k3-merged}. As previously mentioned, Bremner~\cite[Section~3]{bremnersquaresii} showed that his octic K3 surface (corresponding to orbit 3) has maximal geometric Picard rank 20. Through personal communication, the authors were made aware that Damiano Testa and Tony V\'arilly-Alvarado had independently computed the geometric Picard ranks of all the magic octic K3 surfaces in 2014 using the same method of van Luijk.
\end{remark}

Using the geometric construction in Mukai's degree 8 to degree 2 isogeny~\cite{mukai1984moduli,mukai1984symplectic}, the magic octic K3 surfaces naturally determine (singular models of) degree 2 K3 surfaces. We briefly recall this construction. Let $S$ be a degree 8 K3 surface defined by a complete intersection of 3 quadrics in $\P^5= \P(V)$ for a 6-dimensional vector space $V$. The quadrics generate a net $\P^2 \subset \P(\mathrm{Sym}^2V)$. The determinant of the Gram matrix of this net of quadrics is a sextic curve $C$ in $\mathbb{P}^2$. The discriminant double cover associated to this net of quadrics is a degree 2 K3 surface $S' \to \P^2$. When both $S$ and $S'$ are smooth, this construction identifies $S$ as a moduli space of twisted sheaves on $S'$, which is twisted derived equivalent to it, cf.\  \cite{cualduararu2002nonfine},\cite{huybrechts2005equivalences}, \cite[Section~5]{mckinnie2017brauer}.

Although the magic octic K3 surfaces may be singular, the construction still produces intriguing degree 2 K3 surfaces whose arithmetic may reflect that of the original degree 8 models. In our running example $X$ above, the branch locus of the Mukai double cover is given by
\[
xyz(2x + 2y - z)(y-z)(x-z)= 0,
\]
a union of 6 lines, with two pairs of three lines that have a triple intersection. The lines $y = 0$, $y-z = 0$, and $z = 0$ meet at the triple point $[1, 0, 0]$, and the lines $x= 0$, $x - z= 0$, and $z = 0$ meet at the triple point $[0, 1, 0]$. 

The discriminant double covers associated to the magic octic K3 surfaces all have branch sextics decomposing into unions of 6 lines; except in the case of Bremner's octic K3 surface, these sextics admit triple points.  Degree 2 K3 surfaces that are ramified double covers of six lines in general position have been studied before \cite{hosono2021k3, matsumotok3, matsumotomonodromy, kloostermank3}, but degree 2 K3 surfaces that are ramified double covers of six lines arranged with triple points seem less studied. Whether the magic K3 surfaces and their associated degree 2 covers are twisted derived equivalent (possibly after resolution) is an interesting question.

\begin{remark}
K3 surfaces of high geometric Picard rank are known to have a modular transcendental part of their $\ell$-adic cohomology \cite{elkies2013modular, livne1995motivic, costa2025explicit}; exploring these properties in full with the magic K3 surfaces would deepen our understanding of the $\ell$-adic cohomology of $\widetilde{V}$ and its associated Galois representation.
\end{remark} 

In summary, we have the following.

\begin{prop}
\label{prop:K3count}
Up to automorphism, there are 14 magic octic K3 surfaces arising from magic squares with 6 square entries. All of their associated discriminant double covers are branched over arrangements of six lines in $\mathbb{P}^2$.  Except for Bremner's octic K3, which is smooth, all the rest have at worst ordinary double points and their associated discriminant branch sextics contain triple points.
\end{prop}

\begin{proof}
The equations for the magic octic K3 surfaces and their singularities were computed via explicit projection from $V$ in Magma. Code to compute the branch sextic of the associated double cover and to check its number of triple points is provided at \cite{singer_2026}. We note that there is an isomorphism between the K3 surfaces corresponding to orbits $7$ and $14$, as well as between the K3 surfaces corresponding to orbits $2$ and $13$, given in \cite[p.~296]{bremnersquaresii}. All of this information is summarized in Table~\ref{tab:k3-merged}, with the resolution of each K3 denoted as $\widetilde{S}$.
\end{proof}

\begingroup
\setlength{\LTleft}{-40pt plus 1fil}
\setlength{\LTright}{-40pt plus 1fil}
\renewcommand{\arraystretch}{1.15}
\setlength{\tabcolsep}{6pt}
\begin{longtable}{@{}l l c c l@{\hskip 3pt}c@{}}
\toprule
Orbit & Equations & \# Nodes & $\rho(\widetilde{S}_{\overline{\mathbb{Q}}})$ & Branch Sextic Factorization & \# Triple Points \\
\midrule
\endfirsthead

\multicolumn{6}{l}{\textit{Table~\ref{tab:k3-merged} continued from previous page}} \\
\toprule
Orbit & Equations & \# Nodes & $\rho(\widetilde{S}_{\overline{\mathbb{Q}}})$ & Branch Sextic Factorization & \# Triple Points \\
\midrule
\endhead

\midrule
\multicolumn{6}{r}{\textit{Continued on next page}} \\
\endfoot

\bottomrule
\caption{Magic octic K3 surfaces organized by coordinate projection orbit (Figure~\ref{fig:6}), equations in $\mathbb{P}^5$, number of nodes, geometric Picard rank of the resolution, factorization of the branch sextic of the associated double cover, and the number of triple points of the branch sextic.  Orbits are labelled and ordered following the sixteen orbits of Bremner
  \cite{bremnersquaresii}.}
\label{tab:k3-merged}
\endlastfoot

1
  & $\begin{aligned}[t]
       x_0^2 + x_3^2 - 3x_4^2 + x_5^2  &= 0 \\
       x_1^2 - x_3^2 + 2x_4^2 - 2x_5^2 &= 0 \\
       x_2^2 - 2x_4^2 + x_5^2          &= 0
     \end{aligned}$
  & $4$ & $19$
  & $\begin{aligned}[t]
       &xyz(y - z)(x - 2y + z) \\
       &\cdot (2x - 2y + 3z)
     \end{aligned}$
  & $1$ \\
\addlinespace
2, 13
  & $\begin{aligned}[t]
       2x_0^2 - x_3^2 - x_4^2          &= 0 \\
       2x_1^2 - x_3^2 + x_4^2 - 2x_5^2 &= 0 \\
       2x_2^2 + x_3^2 - x_4^2 - 2x_5^2 &= 0
     \end{aligned}$
  & $8$ & $20$
  & $\begin{aligned}[t]
       &xyz(x - y - z)(x - y + z) \\
       &\cdot (x + y)
     \end{aligned}$
  & $2$ \\
\addlinespace
3
  & $\begin{aligned}[t]
       3x_0^2 - 2x_3^2 - 2x_4^2 + x_5^2 &= 0 \\
       3x_1^2 - 2x_3^2 + x_4^2 - 2x_5^2 &= 0 \\
       3x_2^2 + x_3^2 - 2x_4^2 - 2x_5^2 &= 0
     \end{aligned}$
  & $0$ & $20$
  & $\begin{aligned}[t]
       &xyz(x - 2y - 2z) \\
       &\cdot (2x - y + z) \\
       &\cdot (2x + 2y - z)
     \end{aligned}$
  & $0$ \\
\addlinespace
4
  & $\begin{aligned}[t]
       2x_0^2 - x_3^2 - x_5^2          &= 0 \\
       x_1^2 + x_3^2 - 2x_4^2          &= 0 \\
       2x_2^2 + x_3^2 - 2x_4^2 - x_5^2 &= 0
     \end{aligned}$
  & $12$ & $20$
  & $\begin{aligned}[t]
       &xyz(x + y)(2x + y) \\
       &\cdot (4x + 2y - z)
     \end{aligned}$
  & $2$ \\
\addlinespace
5
  & $\begin{aligned}[t]
       x_0^2 - x_2^2 + x_4^2 - x_5^2   &= 0 \\
       x_1^2 + 2x_2^2 - 4x_4^2 + x_5^2 &= 0 \\
       x_3^2 - 2x_4^2 + x_5^2          &= 0
     \end{aligned}$
  & $4$ & $20$
  & $\begin{aligned}[t]
       &xyz(2y - z)(x + y - z) \\
       &\cdot (2x + 4y - z)
     \end{aligned}$
  & $1$ \\
\addlinespace
6
  & $\begin{aligned}[t]
       2x_0^2 - x_4^2 - x_5^2 &= 0 \\
       x_1^2 - 2x_3^2 + x_5^2 &= 0 \\
       x_2^2 - 2x_3^2 + x_4^2 &= 0
     \end{aligned}$
  & $12$ & $20$
  & $xyz(2y - z)(x + y)(2x - z)$
  & $3$ \\
\addlinespace
7, 14
  & $\begin{aligned}[t]
       x_0^2 + x_3^2 - x_4^2 - x_5^2 &= 0 \\
       x_1^2 - 2x_3^2 + x_5^2        &= 0 \\
       x_2^2 - 2x_3^2 + x_4^2        &= 0
     \end{aligned}$
  & $8$ & $19$
  & $\begin{aligned}[t]
       &xyz(x - z)(y - z) \\
       &\cdot (2x + 2y - z)
     \end{aligned}$
  & $2$ \\
\addlinespace
8
  & $\begin{aligned}[t]
       x_0^2 - 2x_4^2 + x_5^2        &= 0 \\
       x_1^2 + x_3^2 - 2x_5^2        &= 0 \\
       x_2^2 - x_3^2 - x_4^2 + x_5^2 &= 0
     \end{aligned}$
  & $8$ & $18$
  & $\begin{aligned}[t]
       &xyz(x + 2z)(x - 2y + z) \\
       &\cdot (x - y)
     \end{aligned}$
  & $2$ \\
\addlinespace
9
  & $\begin{aligned}[t]
       2x_0^2 - x_3^2 - 3x_4^2 + 2x_5^2 &= 0 \\
       x_1^2 + x_4^2 - 2x_5^2           &= 0 \\
       x_2^2 - x_3^2 - x_4^2 + x_5^2    &= 0
     \end{aligned}$
  & $4$ & $18$
  & $\begin{aligned}[t]
       &xyz(x - 2y + z)(2x + z) \\
       &\cdot (2x - 2y + 3z)
     \end{aligned}$
  & $1$ \\
\addlinespace
10
  & $\begin{aligned}[t]
       2x_0^2 + x_3^2 - 2x_4^2 - x_5^2 &= 0 \\
       x_1^2 - 2x_4^2 + x_5^2          &= 0 \\
       2x_2^2 - x_3^2 - x_5^2          &= 0
     \end{aligned}$
  & $8$ & $19$
  & $\begin{aligned}[t]
       &xyz(2y + z)(x - z) \\
       &\cdot (x - 2y + z)
     \end{aligned}$
  & $2$ \\
\addlinespace
11
  & $\begin{aligned}[t]
       x_0^2 - 2x_3^2         +  x_5^2 &= 0 \\
       x_1^2 + 2x_3^2 - x_4^2 - 2x_5^2 &= 0 \\
       x_2^2 - 2x_3^2 + x_4^2          &= 0
     \end{aligned}$
  & $12$ & $19$
  & $xyz(2y - z)(x - y)(x - y + z)$
  & $3$ \\
\addlinespace
12
  & $\begin{aligned}[t]
       2x_0^2 - x_4^2 - x_5^2        &= 0 \\
       x_1^2 - x_3^2 - x_4^2 + x_5^2 &= 0 \\
       2x_2^2 - x_3^2 - x_5^2        &= 0
     \end{aligned}$
  & $8$ & $19$
  & $\begin{aligned}[t]
       &xyz(2y + z)(x - 2y + z) \\
       &\cdot (x + 2y)
     \end{aligned}$
  & $2$ \\
\addlinespace
15
  & $\begin{aligned}[t]
       x_0^2 + x_3^2 - x_4^2 - x_5^2 &= 0 \\
       x_1^2 + x_3^2 - 2x_5^2        &= 0 \\
       2x_2^2 - x_3^2 - x_4^2        &= 0
     \end{aligned}$
  & $8$ & $18$
  & $\begin{aligned}[t]
       &xyz(2y + z)(x + 2z) \\
       &\cdot (x - 2y - 2z)
     \end{aligned}$
  & $2$ \\
\addlinespace
16
  & $\begin{aligned}[t]
       x_0^2 - 2x_3^2 - 2x_4^2 + 3x_5^2 &= 0 \\
       x_1^2 + x_3^2 - 2x_5^2           &= 0 \\
       x_2^2 - 2x_3^2 - x_4^2 + 2x_5^2  &= 0
     \end{aligned}$
  & $4$ & $19$
  & $\begin{aligned}[t]
       &xyz(x + 2y)(2x - 2y + 3z) \\
       &\cdot (2x - y + 2z)
     \end{aligned}$
  & $1$ \\
\end{longtable}
\endgroup
\subsubsection{Product Coordinate Projections: Magic Dodecic K3s} 

We can use products of coordinate projections to $\P^1 \times \P^1 \times \P^1$ to produce magic dodecic (degree 12) K3 surfaces associated (and twisted derived equivalent, as we prove) to each magic octic K3 surface.  In the case of Bremner's octic K3 surface, we prove that the associated magic dodecic K3 surfaces are obtained by resolving the quotient of $Y$ by a Klein four group of symplectic involutions. 

The construction of the magic dodecic K3 surfaces is as follows. After projecting to $\mathbb{P}^5$, which involves a choice of 6 coordinates determined by a choice of orbit (in Figure~\ref{fig:6}), we compose with a dominant rational map
\begin{equation}
\label{eq:MagicK3Projection}
\phi \colon \P^5 \to \mathbb{P}^1 \times \mathbb{P}^1 \times \mathbb{P}^1
\end{equation}
determined by a choice of one of the 15 partitions of the 6 coordinates of $\P^5$ into 3 pairs.  The degeneracy locus is the union of three linear 3-dimensional subspaces of $\P^5$ determined by the partition.  A Magma calculation verified that this degeneracy locus intersects the corresponding magic octic K3 surface in dimension 0.

One such choice of partition gives 
\begin{align*}
\phi : \mathbb{P}^5 &{} \dashrightarrow \mathbb{P}^1 \times \mathbb{P}^1 \times \mathbb{P}^1 \\ 
[A:B:C:F:G:H] &{} \longmapsto ([A:H], [B:G], [C:F])
\end{align*}
whose degeneracy locus is the union of the 3-dimensional subspaces $\{A = H = 0\}$, $\{B = G = 0\}$, and $\{C = F = 0\}$.

To each partition, we consider the composition of dominant rational maps 
\begin{equation}
\label{eq:dodecicmap}
\pi : \P^8 \dashrightarrow \P^5 \dashrightarrow \mathbb{P}^1 \times \mathbb{P}^1 \times \mathbb{P}^1.    
\end{equation}
A Magma computation shows that the image of $\pi$ on $V \subset \P^8$ is a $(2,2,2)$-divisor in $\mathbb{P}^1 \times \mathbb{P}^1 \times \mathbb{P}^1$, whose further image under the Segre embedding $\P^1 \times \P^1 \times \P^1 \hookrightarrow \P^7$ is a degree 12 nodal K3 surface in $\mathbb{P}^7$.  We refer to these as the \linedef{magic dodecic K3 surfaces}; there are 15 for each choice of orbit in Figure~\ref{tab:k3-merged}.

We now explain the relationship between each magic octic K3 surface and its associated magic dodecic K3 surfaces.

\begin{prop}
\label{prop:dodecicK3twistedderivedequivalence}
Let $S$ be a magic octic K3 surface.  Then any associated magic dodecic K3 surface $S'$ has at worst ordinary double point singularities and its minimal resolution is twisted derived equivalent to the minimal resolution of $S$ over~$\C$.
\end{prop}

We recall that K3 surfaces $X$ and $Y$ are \linedef{twisted derived equivalent} if there is an equivalence of derived categories $D^b(X,\gamma) \cong D^b(Y,\delta)$ of sheaves twisted by Brauer classes $\gamma \in \text{Br}(X)$ and $\delta \in \text{Br}(Y)$.  Motivated by work of Orlov,  Huybrechts and Stellari refined and proved a conjecture of C\u{a}ld\u{a}raru, that K3 surfaces $X$ and $Y$ are twisted derived equivalent if and only if there exists a Hodge isometry between rational transcendental cohomology $T(X)_\Q \cong T(Y)_\Q$, see \cite[Chapter~16.4]{huybrechts:K3book}.  Before proving Proposition~\ref{prop:dodecicK3twistedderivedequivalence}, we record a general fact about rational maps of K3 surfaces and twisted derived equivalence. This is known, but lacks a citable reference, so we provide a proof for the convenience of the reader.

\begin{lemma}\label{lem:square-degree-isogeny}
Let $f \colon X \dashrightarrow Y$ be a dominant rational map of generic degree $d = e^2$ between K3 surfaces over $\mathbb{C}$. Then $f$ induces a rational Hodge isometry
\[
T(Y)_{\mathbb{Q}} \xrightarrow{\ \sim\ } T(X)_{\mathbb{Q}},
\]
given by $\tfrac{1}{e} f^{*}$. In particular, $X$ and $Y$ are twisted derived equivalent.
\end{lemma}

\begin{proof}
Since $X$ is a smooth projective surface and $Y$ is smooth and projective, by elimination of indeterminacy for rational maps from a smooth surface to a smooth projective surface \cite[II.7]{beauville1996complex},  there is a composition of blow-ups at points $\alpha \colon \widehat{X} \to X$ such that $\beta := f \circ \alpha \colon \widehat{X} \to Y$ is a morphism. Since $f$ is dominant, $\beta$ is surjective and generically finite of degree $d$.

We first observe that $\alpha^{*}$ identifies the transcendental parts of $X$ and $\widehat{X}$. Since $\alpha$ is a composition of point blow-ups, there is an orthogonal decomposition
\[
H^{2}(\widehat{X},\mathbb{Z}) = \alpha^{*}H^{2}(X,\mathbb{Z}) \oplus
\bigoplus_{i} \mathbb{Z}\,e_{i},
\]
where the $e_i$ are the classes of exceptional divisors of the blowup, with $e_i^2 = -1$ and $e_i \cdot e_j = 0$ for $i \neq j$, and $\alpha^{*}$ is an injective morphism of Hodge structures compatible with the intersection forms. If $v \in T(X)_{\mathbb{Q}}$ and $D \in \mathrm{NS}(\widehat{X})_{\mathbb{Q}}$, then $\langle \alpha^{*}v, D\rangle = \langle v, \alpha_{*}D\rangle = 0$ because $\alpha_{*}D \in \mathrm{NS}(X)_{\mathbb{Q}}$. Hence, $\alpha^{*}T(X)_{\mathbb{Q}} \subseteq T(\widehat{X})_{\mathbb{Q}}$. Conversely, write $w \in T(\widehat{X})_{\mathbb{Q}}$ as $w = \alpha^{*}v + \sum_i c_i e_i$. The classes $e_i$ are algebraic, so pairing with $e_j$ gives $c_j = 0$, and then $\langle v, D \rangle = \langle \alpha^{*}v, \alpha^{*}D\rangle = 0$ for every $D \in \mathrm{NS}(X)_{\mathbb{Q}}$. So, $v \in T(X)_{\mathbb{Q}}$, and $\alpha^{*} \colon T(X)_{\mathbb{Q}} \to T(\widehat{X})_{\mathbb{Q}}$ is an isometry of $\mathbb{Q}$-Hodge structures.

Next, $\beta^{*} \colon H^{2}(Y,\mathbb{Z}) \to H^{2}(\widehat{X},\mathbb{Z})$ is a morphism of Hodge structures, and $\beta_{*}\beta^{*} = d \cdot \mathrm{id}$ by the projection formula. So, we get
\[
\langle \beta^{*}v, \beta^{*}w \rangle_{\widehat{X}} = \langle v, \beta_{*}\beta^{*}w \rangle_{Y} = d \, \langle v, w \rangle_{Y} \qquad \text{for } v,w \in H^{2}(Y,\mathbb{Q}).
\]
In particular, $\beta^{*}$ is injective. Exactly as above, $\langle \beta^{*}v, D\rangle = \langle v, \beta_{*}D\rangle = 0$ for $v \in T(Y)_{\mathbb{Q}}$ and $D \in \mathrm{NS}(\widehat{X})_{\mathbb{Q}}$, so $\beta^{*}T(Y)_{\mathbb{Q}} \subseteq T(\widehat{X})_{\mathbb{Q}} = \alpha^{*}T(X)_{\mathbb{Q}}$.

Therefore,
\[
f^* := (\alpha^{*})^{-1} \circ \beta^{*} \colon T(Y)_{\mathbb{Q}}
\hookrightarrow T(X)_{\mathbb{Q}}
\]
is an injective morphism of $\mathbb{Q}$-Hodge structures satisfying $\langle \varphi v, \varphi w \rangle = d \langle v, w \rangle = e^{2} \langle v, w \rangle$. However, it is a general fact that the transcendental lattice of a K3 surface is an irreducible $\mathbb{Q}$-Hodge structure, see \cite[Chapter~3.3]{huybrechts:K3book}. Since $\varphi(T(Y)_{\mathbb{Q}})$ is a nonzero sub-Hodge structure of $T(X)_{\mathbb{Q}}$, it is all of $T(X)_{\mathbb{Q}}$, so that $\varphi$ is an Hodge isomorphism, and then $\tfrac{1}{e}f^*$ is a Hodge isometry. Finally, K3 surfaces are twisted derived equivalent if and only if there is a rational Hodge isometry between the transcendental parts of their Hodge structures, see \cite[Chapter~16.4]{huybrechts:K3book}.
\end{proof}

\begin{proof}[Proof of Proposition \ref{prop:dodecicK3twistedderivedequivalence}]
Let $S \subset \P^5$ be a choice of magic octic K3 surface, let $S' \subset \P^1 \times \P^1 \times \P^1$ denote a choice of dodecic K3 surface associated to $S$, and let $p \colon S \dashrightarrow S'$ denote the dominant rational map defined by restricting $\phi$ from \eqref{eq:MagicK3Projection} to~$S$.  This is summarized as
\begin{align*}
\P^8 &{} \dashrightarrow  \P^5  \dashrightarrow \mathbb{P}^1 \times \mathbb{P}^1 \times \mathbb{P}^1 \\
V &{} \dashrightarrow S  \dashrightarrow S'
\end{align*}

A computation in Magma \cite{singer_2026} shows that $S$ and $S'$ have at worst nodal singularities. 
For $S'$, this uses the fact that $S'$ is a hypersurface in the threefold $\mathbb{P}^{1} \times \mathbb{P}^{1} \times \mathbb{P}^{1}$, so a singular point is an ordinary double point precisely when the $3 \times 3$ Hessian of a local equation is nonsingular there. Ordinary double points are resolved crepantly, so the minimal resolutions $\widetilde{S} \to S$ and $\widetilde{S}' \to S'$ are smooth projective K3 surfaces.

The base locus of $p \colon S \dashrightarrow S'$ is contained in the degeneracy locus of $\phi$, namely the intersection of $S$ with the three $3$-dimensional linear subspaces whose union is that degeneracy locus; as mentioned above, a Magma computation shows this is a finite set of points. Hence $p$ induces a rational map
\[
f \colon \widetilde{S} \dashrightarrow \widetilde{S}'
\]
between smooth projective K3 surfaces, which is dominant of the same generic degree as $p$. By Lemma~\ref{lem:square-degree-isogeny}, we are reduced to proving that $p$ is generically of square degree.  In fact, we prove that $p$ is generically of degree 4.

Write the three defining quadrics of $S$ as $Q_{1}, Q_{2}, Q_{3}$ and group each into the parts supported on the three pairs of chosen coordinates,
\[
Q_{j} = q_{j1} + q_{j2} + q_{j3},
\]
so that $q_{jk}$ is a diagonal binary quadratic form in the two coordinates of the $k$th pair. Let $T = \mathbb{G}_{m}^{3}/\mathbb{G}_{m}$ act on $\mathbb{P}^{5}$ by scaling each pair independently. Away from the degeneracy locus of $\phi$, the fibers of $\phi$ are exactly the $T$-orbits, so the fiber of $p$ through a point $x$ is $(T \cdot x) \cap S$.

For $x \in S(\C)$ with all coordinates nonzero, let
\[
N(x) = \bigl( q_{jk}(x) \bigr)_{j,k} \in \mathrm{Mat}_{3\times 3}(\C).
\]
Each row of $N(x)$ sums to $Q_{j}(x) = 0$, so $(1,1,1)^{\mathsf{T}} \in \ker N(x)$ and ${\mathrm{rank}\,N(x) \leq 2}$. Suppose $\mathrm{rank}\,N(x) = 2$, so that $\ker N(x)$ is spanned by $(1,1,1)^{\mathsf{T}}$. If $t = (\lambda_{1}, \lambda_{2}, \lambda_{3}) \in T$ satisfies $t \cdot x \in S(\C)$, then setting $\mu = (\mu_1,\mu_2,\mu_3) = (\lambda_{1}^{2}, \lambda_{2}^{2}, \lambda_{3}^{2})$ we have 
\[
0 = Q_{j}(t \cdot x) = \sum_{k=1}^{3} \mu_{k}\, q_{jk}(x)
\qquad \text{for } j = 1,2,3,
\]
so $\mu \in \ker N(x)$, and is thus a nonzero multiple of $(1,1,1)$. Up to rescaling the $Q_j$ and $t$ by global scalars, we may assume that $\mu = (1,1,1)$, i.e.\ $\lambda_{k} = \pm 1$ for each $k$. Thus, the fiber of $p$ at $x$ is the orbit of $x$ under $G \cong (\mu_{2})^{3}/\mu_{2}$, which has order $4$. This action is free at $x$; indeed, the fixed locus in $\mathbb{P}^{5}$ of a nontrivial sign change is the union of the two linear subspaces cut out by the vanishing of the coordinates it negates, respectively of the coordinates it fixes, and $x$ lies on neither since all of its coordinates are nonzero.

It remains to check that $\mathrm{rank}\, N = 2$ at a general point of $S$. Since $S$ is irreducible and the condition $\mathrm{rank}\,N \leq 1$ is closed, this amounts to exhibiting, for each magic octic K3 surface and each of the choices of pairs, a $2 \times 2$ minor of $N$ that does not lie in the ideal of $S$. This is a direct computation, see \cite{singer_2026}.
\end{proof}

We wonder whether for any given choice of magic octic K3 surface $S$ and any choice of associated magic dodecic K3 surface $S'$, the rational map $S \dashrightarrow S'$ yields an isogeny of the full Hodge structure.

In the case of Bremner's octic K3 surface $Y$, which is already smooth, we can see the above as an byproduct of the existence of symplectic involutions.

\begin{prop}
\label{prop:dodecicK3Nikulin}
Let $Y$ be Bremner's octic K3 surface.  Then, the resolutions of each of the corresponding magic dodecic K3 surfaces are isomorphic to the resolution of the quotient of $Y$ by a Klein-four group acting by symplectic involutions.
\end{prop}

Indeed, the resolution of the quotient of a smooth K3 surface by a finite group generated by symplectic involutions is itself a smooth K3 surface \cite{nikulin1980finite}. Hence, the images of the projections $Y \to \mathbb{P}^1 \times \mathbb{P}^1 \times \mathbb{P}^1$ are alternative singular models of the quotient of $Y$ by finite groups generated by symplectic involutions.

\begin{proof}
Let $Y'$ be a magic dodecic K3 surface associated to Bremner's octic K3 surface with projection $p\colon Y \dashrightarrow Y'$ and resolution $\widetilde{Y'}$. Since the generic fiber of $p$ contains the orbit of a point under the group $G \cong (\mu_2)^3/\mu_2$ of sign changes on the pairs as in the proof of Proposition \ref{prop:dodecicK3twistedderivedequivalence}, the projection $p\colon Y \dashrightarrow Y'$ factors through the quotient rational map $\pi'\colon Y \dashrightarrow Y/G$ as $p=\overline{p}\circ\pi'$ for a dominant rational map $\overline p\colon Y/G\dashrightarrow Y'$. We let $Z$ denote $Y/G$. Then $Z$ has nodes at the images of the degeneracy points of $p$, where $G$ acts non-freely, and we let $\widetilde{Z}$ denote its resolution.  Now, by the proof of Proposition~\ref{prop:dodecicK3twistedderivedequivalence}, $p$ has generic degree~4. On the other hand, $\pi'$ has degree 4, so that $\overline{p}$ has degree 1. Hence, $\overline{p}$ is birational. This implies that the induced map on resolutions $\widetilde{\overline{p}} \colon \widetilde{Z} \to \widetilde{Y'}$ is birational. 

To conclude, we need to show that $G$ acts by symplectic involutions on $Y$. By Nikulin's theorem, this will show that $\widetilde{Z}$ is a smooth K3 surface. Since $\widetilde{Z}$ is a smooth K3 surface and a birational map between smooth minimal surfaces of non-negative Kodaira dimension must be an isomorphism, $\widetilde{\overline{p}}$ is an isomorphism.

Denote the generators of $G$ as $\sigma_1$ and $\sigma_2$. As the statement is geometric, we will assume that $Y$ is over $\mathbb{C}$. Denote by $Q_1$, $Q_2$, and $Q_3$ the diagonal quadratic forms generating the ideal of $Y$ and by $z_0, \dotsc, z_5$ a set of coordinates on $\mathbb{C}^6$. By the adjunction formula, we can describe the global symplectic 2-form on $Y$ as the restriction to $Y$ of the descent of the iterated Poincar\'e residue
\[
    \omega = \mathrm{Res}\bigg ( \frac{\Omega}{Q_1Q_2Q_3} \bigg ), \qquad \Omega = \sum_{i = 0}^5 (-1)^i z_i dz_0 \wedge \cdots \wedge \widehat{dz_i} \wedge \cdots \wedge dz_5
    \]
    to $\mathbb{P}^5$. Explicitly, on an open subset $U$ of $\mathbb{P}^5$, we have
    \[
    \bigg (dQ_1 \wedge dQ_2 \wedge dQ_3 \wedge \widetilde{\omega} \bigg )\bigg|_U = \Omega|_U,
    \]
where $\widetilde{\omega}|_Y = \omega$. Since the $Q_i$'s are diagonal quadratic forms, coordinate sign-changes do not affect $dQ_1$, $dQ_2$, or $dQ_3$ under pullback. Hence, the pullback of $\Omega$ under $\sigma_1$ and $\sigma_2$ will determine the pullback of $\widetilde{\omega}$, which in turn determines the pullback of $\omega$ on $Y$ by $\sigma_1|_Y$ and $\sigma_2|_Y$. By the explicit formula for $\Omega$, we see that $\sigma_2^*(\Omega) = \sigma_1^*(\Omega) = (-1)^2\Omega = \Omega$. This implies $\sigma_2|_Y^*(\omega) = \sigma_1|_Y^*(\omega) = \omega$, so that $\sigma_1|_Y$ and $\sigma_2|_Y$ are symplectic involutions of $Y$.
\end{proof}

\subsubsection{Product Coordinate Projections: Magic Enriques Surfaces}

Finally, we can use coordinate projections to $\P^2 \times \P^2$ to produce other magic surfaces admitting dominant rational maps from the magic octic K3 surfaces.  In the case of Bremner's octic K3 surface $Y$, these turn out to be Enriques surfaces with $Y$ as their K3 double cover. 

The construction of these magic surfaces proceeds in a similar fashion as for the magic dodecic K3 surfaces. After projecting to $\mathbb{P}^5$, which involves a choice of the 6 coordinates determined by a choice of orbit (in Figure~\ref{fig:6}), and determines a magic octic K3 surface, we compose with a dominant rational map 
\[
\psi \colon \P^5 \to \mathbb{P}^2 \times \mathbb{P}^2
\]
determined by a choice of one of the 10 partitions of the 6 coordinates of $\P^5$ into disjoint triples.  The degeneracy locus of $\psi$ is the union of two disjoint linear 2-dimensional subspaces $\Pi_1 \cup \Pi_2 \subset \P^5$ determined by the partition.  A Magma calculation verifies that this degeneracy locus intersects the corresponding magic octic K3 surface in at most dimension 0.

Hence we consider the composition of dominant rational maps 
\begin{equation*}
\label{eq:dodecicmap2}
\pi : \P^8 \dashrightarrow \P^5 \dashrightarrow \mathbb{P}^2 \times \mathbb{P}^2.
\end{equation*}
The geometry of the surface that is the image of $\pi$ restricted to $V \subset \P^8$ depends on the choice of magic octic K3 surface and also on the choice of partition.

To each partition (of the coordinates of $\P^5$ into disjoint triples) there is an associated linear involution $\sigma : \P^5 \to \P^5$ changing the signs on the three coordinates in one of the triples of the partition.  While there are ${6 \choose 3} = 20$ involutions changing the signs of three coordinates, each is projectively equivalent to the involution changing the signs of the complementary three coordinates, so we only have $20/2 = 10$ unique such involutions, which are in bijection with partitions.  The fixed-point locus of $\sigma$ on $\P^5$ is the same disjoint union of planes as the degeneracy locus of $\psi$. 
As an example, we can choose
\begin{align*}
\psi \colon \mathbb{P}^5 &{} \dashrightarrow \mathbb{P}^2 \times \mathbb{P}^2 \\
[A,B,C,F,G,H] &{} \longmapsto ([A,B,C], [F, G, H])
\end{align*}
whose degeneracy locus is the union on disjoint planes $\{A = B = C = 0\}$ and $\{F = G = H = 0\}$, and with corresponding involution 
\begin{align*}
\sigma \colon \mathbb{P}^5 &{} \longrightarrow \mathbb{P}^5 \\
[A,B,C,F,G,H] &{} \longmapsto [A,B,C, -F, -G, -H].
\end{align*}

In the case of Bremner's octic K3 surface $Y \subset \P^5$ (corresponding to orbit 3 in Figure~\ref{fig:6}), we can precisely identify the geometry of the image surface.

\begin{lemma}
Let $Y \subset \P^5$ be Bremner's octic K3 surface.  For a fixed partition of the coordinates of $\P^5$ into disjoint triples, consider the associated dominant rational map $\psi \colon \P^5 \to \P^2 \times \P^2$, image surface $\psi(Y)=S$, and involution $\sigma$.  Then $S$ is an Enriques surface and its image $S \subset \P^8$, under the Segre embedding $\P^2 \times \P^2 \to \P^8$, has degree 16.  The rational map $\psi$ restricts to a morphism $\psi \colon Y \to S$, and the involution $\sigma$ restricts to a fixed-point free involution $\sigma \colon Y \to Y$.      
\end{lemma}
\begin{proof}
A direct Magma computation verifies the fact that, for any choice of partition, $S$ is a (smooth) Enriques surface and computes its degree under the Segre embedding.

To verify the remaining claims, we need to check that the degeneracy locus of $\psi$ on $\P^5$, equivalently, the fixed-point locus of $\sigma$, is disjoint from $Y$.  To this end, write the three defining quadrics of $Y$ as $Q_j=q_{j1}+q_{j2}$, where $q_{jk}$ is a diagonal quadratic form in the coordinates of the $k$th triple in the chosen partition, and let $M_k\in\operatorname{Mat}_3(\mathbb{Q})$ be the matrix expressing $(q_{1k},q_{2k},q_{3k})$ in the squares of those coordinates. A direct computation shows $\det M_k\neq 0$ for $k=1,2$.

Let $\Pi_1,\Pi_2\subset\mathbb{P}^5$ be the planes cut out by the vanishing of the first, resp.\ second, triple of coordinates.  Then $\Pi_1\cup\Pi_2$ is the degeneracy locus of $\psi$ and the fixed locus of $\sigma$ on $\P^5$. If $x\in Y\cap\Pi_1$ then $q_{j1}(x)=0$ for all $j$, hence $q_{j2}(x)=Q_j(x)-q_{j1}(x)=0$ for all $j$, and $\det M_2\neq 0$ forces the remaining three coordinates of $x$ to vanish as well. So, $Y\cap\Pi_1=\varnothing$, and symmetrically $Y\cap\Pi_2=\varnothing$. 
\end{proof}

In fact, in the case of Bremner's octic K3 surface, the multihomogeneous ideal of each Enriques surface $S$ is generated by three $(2,2)$-forms in $\P^2 \times \P^2$. We refer to these 10 surfaces as the \linedef{magic Enriques surfaces}.  The projection of $S \subset \P^2 \times \P^2$ to one of the factors $\P^2$ is a Cossec--Verra polarization of the Enriques surface $S$, see \cite{cossec,verra}, cf.\ \cite[Section~3.2]{liedtke}, \cite[Chapter~3.4.2]{cossec_dolgachev_liedke}.

Finally, we prove the following.

\begin{prop}
Each magic Enriques surface $S$ has Bremner's octic K3 surface $Y$ as its K3 double cover $\psi \colon Y \to S$.
\end{prop}
\begin{proof}
Composing $\psi \colon Y \to S$ with the Segre embedding $\mathbb{P}^2\times\mathbb{P}^2 \hookrightarrow\mathbb{P}^8$, we get that $\psi^{*}\mathscr{O}_S(1,1)\cong \mathscr{O}_Y(2)$. This is ample, so $\psi \colon Y \to S$ contracts no curves and is therefore a finite morphism. Moreover, we have
\[
\bigl(\psi^{*}\mathscr{O}_S(1,1)\bigr)^{2} = 4\deg Y = 32,
\]
and since $\deg S=16$, we have that $\psi$ is finite of degree $2$.

By the computation in the proof of Proposition~\ref{prop:dodecicK3Nikulin}, applied now to the associated involution $\sigma$, changing the signs of three coordinates, we have $\sigma^{*}\Omega =(-1)^{3}\Omega=-\Omega$. Since the $Q_j$ are diagonal, this gives $\sigma|_Y^{*}\omega=-\omega$. As $\sigma \colon Y \to Y$ is free, the quotient $Z = Y/\sigma$ is a smooth Enriques surface with $q : Y\to Z$ its K3 double cover.

Finally, $\psi$ is $\sigma$-invariant, so it factors as $\psi=\overline{\psi}\circ q$ and $\overline{\psi}\colon Z\to S$ is finite. Then $2=\deg\psi=\deg q \cdot\deg\overline{\psi}=2\deg\overline{\psi}$, so $\overline{\psi}$ is birational. Since $S$ is smooth, and a finite birational morphism onto a normal variety in characteristic zero is an isomorphism, we have that $S\cong Z$ so that $Y$ is the K3 double cover of $S$.
\end{proof}

\section{The Picard Lattice of $V$}
\label{sec:PicardLattice}
In this section, we establish the bounds on $\widetilde{V}$'s geometric Picard rank stated in Theorem~\ref{thm:picard-rank}. We exhibit an explicit collection of $1204$ divisor classes: $256$ exceptional divisors arising from the resolution of the $A_1$ singularities, $416$ irreducible components of split hyperplane sections defined over various number fields, $384$ preimages of lines on quartic del Pezzo surfaces admitting a projection from $V$, and $148$ preimages of lines on cubic surfaces admitting a projection from $V$. We show, via a direct computation of their intersection matrix, that they span a sublattice of rank $518$ in $\Pic(\widetilde{V}_{\overline{\Q}})$. We begin in Proposition~\ref{prop:pic-ns} by verifying that $\Pic(\widetilde{V}_{\overline{\Q}}) \cong \NS(\widetilde{V}_{\overline{\Q}})$ is torsion-free. We then provide a summary of all of the divisors produced in sections 3 and 4. Finally, we assemble the intersection-theoretic data, carefully accounting for the contribution of the $A_1$ singularities to the intersection pairing on $\widetilde{V}$. This completes the proof of Theorem~\ref{thm:picard-rank} with the help of Magma~\cite{bosma1997magma}. 

To begin, it is known that linear and numerical equivalence agree for divisors on complete intersections. However, since we are working with a resolution of a complete intersection, it is important to spell out a proof:

\begin{prop}
\label{prop:pic-ns}
    $\mathrm{Pic}(\widetilde{V}_{\overline{\mathbb{Q}}}) \cong \mathrm{NS}(\widetilde{V}_{\overline{\mathbb{Q}}})$. Furthermore, $\mathrm{NS}(\widetilde{V}_{\overline{\mathbb{Q}}})$ is torsion-free.
\end{prop}

\begin{proof}
    By GAGA \cite{serre_1956_gomtrie} and the fact that N\'eron-Severi and Picard groups are insensitive to algebraically closed extension of algebraically closed fields \cite{maulik2012neron}, the Picard group of $\widetilde{V}^{an}$ and the geometric Picard group of $\widetilde{V}$ agree. So, we can use the exponential exact sequence in the analytic category    
    \[
    0 \longrightarrow 2\pi i\mathbb{Z} \longrightarrow \mathscr{O}_{\widetilde{V}^{an}} \longrightarrow \mathbb{G}_{m} \longrightarrow 0
    \]    
    and extend to the long exact sequence of cohomology. Since $h^{1, 0} = 0$ by Theorem~\ref{thm:main1}, the long exact sequence decomposes to the exact sequence    
    \[
    0 \longrightarrow \mathrm{Pic}(\widetilde{V}^{an}) \longrightarrow H^2(\widetilde{V}^{an}, \mathbb{Z}) \longrightarrow H^2(\widetilde{V}^{an}, \mathscr{O}_{\widetilde{V}^{an}}) \longrightarrow H^2(\widetilde{V}^{an}, \mathbb{G}_m) \longrightarrow \cdots
    \]
    This shows that $\mathrm{Pic}(\widetilde{V}^{an})$ injects into $H^2$, so $\mathrm{Pic}^0(\widetilde{V}^{an})$ is trivial and $\mathrm{Pic}(\widetilde{V}^{an}) \cong \mathrm{NS}(\widetilde{V}^{an})$. Since $H^2(\widetilde{V}^{an}, \mathbb{Z})$ is torsion-free by Theorem~\ref{thm:main1}, $\mathrm{NS}(\widetilde{V}^{an})$, is torsion-free.
    
\end{proof}

In total, counting the 416 split hyperplane sections we constructed in Sections~\ref{subsec:nondistinct}, \ref{subsec:coordinate_hyperplane_sections}, and \ref{sec:numberfields}, the 384 preimages of lines on magic quartic del Pezzo surfaces constructed in Section \ref{sec:magic-surfaces}, the 148 preimages of lines on magic cubic surfaces constructed in Section \ref{sec:magic-surfaces}, and the 256 exceptional classes, we have produced 1204 divisors on $\widetilde{V}$. They are distributed as follows:

\begin{center}
\label{tab:divisorcount}
\begin{tabular}{c|c}
    Method & Number of Divisors \\
    \hline
    Irreducible Component of $Z$ over $\mathbb{Q}$ &  176 \\
    Irreducible Component of $Z'$ over $\mathbb{Q}(i, \sqrt2)$& 64\\
    Irreducible Component of $Z''$ over $\mathbb{Q}(i, \sqrt{2}, \sqrt3)$ &  176\\
     Preimages of Lines on dP4s over $\mathbb{Q}(i, \sqrt2, \sqrt3, \sqrt5)$ & 384 \\
     Preimages of Lines on Cubic Surfaces over $\mathbb{Q}(i, \sqrt2, \sqrt3, \sqrt5)$ & 148 \\
    Exceptional Divisors over $\mathbb{Q}(i, \sqrt{2})$ & 256 \\
    \hline 
    Total & 1204 \\
\end{tabular}
\end{center}

In order to understand the sublattice of the N\'eron--Severi group generated by these 1204 divisors, we compute the intersection matrix between them.  The rank of this symmetric matrix then coincides with the rank of the sublattice generated by these divisors.   
Before computing the intersection matrix, we need to give some details on how intersection theory behaves on $V$ and $\widetilde{V}$. 

Let $S$ be a surface with an $A_1$ singularity $p$ and $\pi : \widetilde{S} \to S$ the resolution of $S$ at $p$ with exceptional divisor $E$. For a curve $C \subset S$ the divisor class of $\pi^*C$ is given by
\[
\pi^*C = \widetilde{C} + aE
\]
where $\widetilde{C}$ is the strict transform, and $a$ is some integer. This class is subject to the relation $\pi^*C.E = 0$, meaning that $\widetilde{C}.E + aE^2 = 0$. Denoting the multiplicity of $C$ at $p$ as $m_p(C)$, this yields $\widetilde{C}.E = m_p(C) = 2a$ since $E^2 = -2$. So, $a = \frac{m_p(C)}{2}$. Fractions are admissible since we work with Weil divisors that may not be Cartier on $S$. Furthermore, for any curves $C$ and $D$ on $S$ that intersect transversally, their intersection can be computed as
\[
C.D  = \frac{m_p(C)m_p(D)}{2} + \widetilde{C}.\widetilde{D}
\]
via the above formula.

For self-intersection, we recall that the self-intersection formula for a curve $C$ on a smooth surface $S$ is 
\[
C^2 = 2g(C) - 2 - K_S.C
\]
by adjunction formula.  When $S$ has $A_1$-singularities, this remains the same when we consider the self-intersection of the strict transform of a curve $C \subset S$ to the resolution $\pi: \widetilde{S} \to S$. 

However, computing the intersection numbers for our list of divisors is easier than the above formula suggests. We verified that whenever two distinct curves in our list of 1204 intersect nontrivially, their scheme-theoretic intersections are transverse. Each curve was also verified to meet the singular points with multiplicity one. Moreover, at each node $p$ lying on two of our curves $C$ and $D$, we verified that $C$ and $D$ have distinct tangent directions at $p$. Equivalently, the strict transforms $\widetilde{C}$ and $\widetilde{D}$ meet the exceptional divisor $E_p$ in distinct points. 

Hence, $\widetilde{C}$ and $\widetilde{D}$ are disjoint over the singular locus  $\mathrm{Sing}(V)$ and meet transversally elsewhere, so each point of $C \cap D \smallsetminus \mathrm{Sing}(V)$ contributes $1$ to $\widetilde{C}. \widetilde{D}$, and each node of $C \cap D$ contributes $0$. We thus can compute
\[
\widetilde{C}. \widetilde{D} = \#(C\cap D) - \#(C \cap D \cap \mathrm{Sing}(V)).
\]

For the self-intersection, we know that $K_{\widetilde{V}} = 3H$. Additionally, since each curve meets the singular locus transversally with multiplicity one, blowing up the $A_1$-points does not alter the arithmetic genus of the strict transform. From this and the adjunction formula above, we derive that for a curve $C$ on $V$, the self-intersection of its strict transform is 
\[
\widetilde{C}^2 = 2g(\widetilde{C}) - 2 - 3\deg(\widetilde{C}) = 2g(C) - 2 - 3\deg(C),
\]
where $g(C)$ is the arithmetic genus. Using this, we compute that the $1204 \times 1204$ intersection matrix of the strict transforms of the $1204$ curves we have produced has rank $518$.  This sublattice contains $H$, since it contains all the components of various hyperplane sections.  By the Hodge index theorem, the orthogonal complement of $H$ is negative definite, and thus by a lattice-theoretic argument, the rank of the Gram matrix of the collection of divisors coincides with the rank of the sublattice it generates. Hence, $\rho(\widetilde{V}_{\overline{\mathbb{Q}}}) \geq 518$. 

More precisely, the $948$ divisors given by split hyperplane sections and preimages of lines contribute rank $262$, as the exceptional classes give $256$ independent elements. Theorem~\ref{thm:main1} and the Lefschetz theorem on $(1, 1)$-classes then give the upper bound $\rho(\widetilde{V}_{\overline{\mathbb{Q}}})\leq 544$. All computational data and Magma code used to construct the divisors and verify the intersection-theoretic computations are available in the accompanying repository \cite{singer_2026}. In particular, the repository contains the complete $1204 \times 1204$ matrix and code to reproduce the rank computation.

\begin{remark}
We expect the geometric Picard rank of $\widetilde{V}$ to be maximal. Indeed, any new nontrivial divisor classes will yield many more via the action of the automorphism group. The methods of Sections 3 and 4 appear to be saturated: every divisor produced by our final constructions, including lifts of divisors from the magic K3 surfaces, already appeared in the span of the 1204 classes above. We therefore expect producing the remaining classes to require Hodge-theoretic methods. The sublattice $L$ spanned by our 1204 classes is stable under the automorphism group $G$, and the intersection form is nondegenerate on $L$, so $L^\perp \cap NS(\widetilde{V}_{\overline{\mathbb{Q}}})$ is a $G$-stable complement. If $\rho_{\widetilde{V}_{\overline{\mathbb{Q}}}}$ is maximal, this complement is a 26-dimensional representation of $G$, compatibly with the Galois action.
\end{remark}

\appendix
\section{Lines on $V$}
\label{subsec:lines-conics}

In this appendix, we show that $V$'s Fano scheme of lines is empty, meaning that $V$ contains no lines over $\Qbar$ or $\mathbb{C}$. The algorithm used for these calculations builds on a method of Elsenhans--Jahnel~\cite{elsenhans2008k3} via Schubert cells of Grassmannians to parametrize the Fano scheme of lines. To describe this approach, let $\mathbb{G}(1,8)$ denote the projective Grassmannian of lines in $\mathbb{P}^8$, and let $F(V) \subset \mathbb{G}(1,8)$ denote the Fano scheme of lines on $V$. For a single quadric $Q \subset \mathbb{P}^8$, the Fano scheme $F(Q)$ is cut out in $\mathbb{G}(1,8)$ by a section of $\mathscr{E} = \mathrm{Sym}^2 (\mathscr{S}^\vee)$, where $\mathscr{S}$ is the tautological rank-2 subbundle on $\mathbb{G}(1,8)$. The bundle $\mathscr{E}$ has rank $3$, so $F(Q)$ has expected codimension $3$. Since $V$ is a complete intersection of six quadrics $Q_1, \ldots, Q_6$, the Fano scheme of lines is the intersection
\[
F(V) = F(Q_1) \cap \cdots \cap F(Q_6) \subset \mathbb{G}(1,8),
\]
and has expected codimension $6 \cdot 3 = 18$ in a Grassmannian of dimension $2 \cdot 7 = 14$. The expected dimension is therefore negative, so we expect $F(V) = \varnothing$. This turns out to be true:

\begin{prop}
\label{prop:EmptyLines}
The Fano scheme of lines $F(V)$ of $V$, the variety of $3 \times 3$ magic squares of squares, is empty.
\end{prop}
 
\begin{proof}
The idea is to stratify $\mathbb{G}(1,8)$ by Schubert cells, parametrize lines in each cell by an explicit affine chart, and verify on each chart that the ideal cut out by the six quadrics has no solutions using Gr\"obner bases. Fix the standard complete flag $\mathbb{P}^0 \subseteq \mathbb{P}^1 \subseteq \cdots \subseteq \mathbb{P}^8$ in $\mathbb{P}^8$. The Schubert cells of $\mathbb{G}(1,8)$ are indexed by partitions $\sigma_{a_1, a_2}= (a_1, a_2)$ with $7 \geq a_1 \geq a_2 \geq 0$; there are $36$ such cells with respect to our given flag. The open cell $\sigma_{0,0}$ parametrizes lines avoiding $\mathbb{P}^6 = \mathrm{span}(e_2, \dots, e_8)$ and is isomorphic to $\mathbb{A}^{14}$. The closed cells have lower dimension.
 
On the open cell $\sigma_{0,0}$, a line is uniquely written as the row span of a matrix
\[
\begin{pmatrix}
1 & 0 & x_0 & x_1 & x_2 & x_3 & x_4 & x_5 & x_6 \\
0 & 1 & y_0 & y_1 & y_2 & y_3 & y_4 & y_5 & y_6
\end{pmatrix},
\]
giving an isomorphism $\sigma_{0,0} \cong \mathbb{A}^{14}$. A line of this form lies on $V$ if and only if each defining quadric $Q_i$ vanishes on the parametrized family 
\[
[s : t : s x_0 + t y_0 : \cdots : s x_6 + t y_6]
\]
for all $[s:t] \in \mathbb{P}^1$. Writing $Q_i(s,t)$ as a homogeneous polynomial of degree $2$ in $(s,t)$ with coefficients in $\mathbb{Z}[x_0, \ldots, x_6, y_0, \ldots, y_6]$, we obtain three equations per quadric (the coefficients of $s^2$, $st$, $t^2$), for a total of $18$ equations in $14$ variables, defining the restriction $F(V) \cap \sigma_{0,0}$ as a subscheme of $\mathbb{A}^{14}$. A Gröbner basis computation in Magma~\cite{bosma1997magma} shows that this ideal
contains $1$, so $F(V) \cap \sigma_{0,0} = \varnothing$. Analogous charts and calculations cover each of the remaining $35$ Schubert cells. Since $F(V) \cap \sigma = \varnothing$ for every Schubert cell $\sigma$, we conclude $F(V) = \varnothing$, so $V$ contains no lines. The full Magma code, including the explicit affine charts on each Schubert cell and the Gröbner basis verifications, is available at \cite{singer_2026}.
\end{proof}
 
\begin{remark} This trivially implies that no lines meet the distinct locus $U$. We record it because the question of low-degree rational curves on $U$ --- rather than on all of $V$ --- is what is directly relevant to the Diophantine problem. By contrast, $V$ contains many conics. The 128 rational conics arising from the 3-distinct-entry locus $Z_3$ in Section~\ref{subsec:nondistinct} are all contained in the closed locus $Z_3 \subset V \smallsetminus U$.  We leave open the question of whether these are all the conics contained in $V$; we know there are finitely many by \cite{bruin_2022_explicit}.
\end{remark}

\bibliographystyle{amsplain}
\bibliography{Squares.bib}

\end{document}